\documentclass{article}
\usepackage[margin=1in]{geometry}
\usepackage[english]{babel}
\usepackage[utf8]{inputenc}
\usepackage{subcaption}
\usepackage{amsmath}
\usepackage{amssymb}
\usepackage{amsfonts}
\usepackage{amsthm}
\usepackage{comment}
\usepackage{hyperref}
\hypersetup{colorlinks=true,linkcolor=blue,citecolor=blue,urlcolor=blue} 
\usepackage{mathrsfs,mathtools}
\usepackage[all]{xy}
\usepackage[pdftex]{graphicx}
\usepackage{color}
\usepackage{cite}
\usepackage{url}
\usepackage{indent first}
\usepackage[labelfont=bf,labelsep=period,justification=raggedright]{caption}
\usepackage[english]{babel}
\usepackage[utf8]{inputenc}
\usepackage[colorinlistoftodos]{todonotes}
\usepackage{tkz-fct}
\usepackage{tikz}
\usetikzlibrary{calc}
\usepackage{multicol}
\PassOptionsToPackage{dvipsnames,svgnames}{xcolor}
\usepackage{textcomp}

\newcommand{\eps}{\varepsilon}

\newcommand{\wt}[1]{\widetilde{#1}}

\newcommand{\ux}{\underline{x}}
\newcommand{\uL}{\underline{\texttt{L}}}
\newcommand{\tL}{\texttt{L}}
\newcommand{\GG}{\mathcal{G}}
\newcommand{\bG}{\boldsymbol{G}}
\newcommand{\bGG}{\boldsymbol{\mathcal{G}}}
\newcommand{\bE}{\boldsymbol{E}}

\newcommand{\ba}{\boldsymbol{a}}
\newcommand{\bp}{\boldsymbol{p}}
\newcommand{\bL}{\textbf{L}}
\newcommand{\buL}{\underline{\textbf{L}}}
\renewcommand{\P}{\mathbb{P}}
\newcommand{\R}{\mathbb{R}}
\newcommand{\E}{\mathbb{E}}
\usepackage{dsfont}
\newcommand{\one}{\mathds{1}}
\newcommand{\FF}{\mathscr{F}}
\newcommand{\de}{\mathrm{d}}
\newcommand{\bPhi}{\boldsymbol{\prescript{\star}{}{\Phi}}}
\newcommand{\lbd}{\textnormal{lbd}}
\newcommand{\ubd}{\textnormal{ubd}}

\DeclareMathOperator{\Var}{Var}

\theoremstyle{plain}
\newtheorem{thm}{Theorem}
\newtheorem{lemma}[thm]{Lemma}

\newtheorem{prop}[thm]{Proposition}

\theoremstyle{definition}
\newtheorem{defn}[thm]{Definition}

\newtheorem{remark}[thm]{Remark}
\theoremstyle{remark}

\numberwithin{equation}{section}
\numberwithin{thm}{section}

\title{Upper bounds on the $2$-colorability threshold of random $d$-regular $k$-uniform hypergraphs for $k\geq 3$}
\author{Evan Chang \thanks{High Technology High School. Email: \textup{\tt evchang@ctemc.org}}, Neel Kolhe  \thanks{Lynbrook High School. Email: \textup{\tt neel@kolhe.org}}, Youngtak Sohn \thanks{Department of Mathematics, Massachusetts Institute of Technology. Email: \textup{\tt youngtak@mit.edu}}}
 \date{July 13, 2023}
\begin{document}

\maketitle

\begin{abstract}
For a large class of random constraint satisfaction problems (\textsc{csp}), deep but non-rigorous theory from statistical physics predict the location of the sharp satisfiability transition. The works of Ding, Sly, Sun (2014, 2016) and Coja-Oghlan, Panagiotou (2014) established the satisfiability threshold for random regular $k$-\textsc{nae-sat}, random $k$-\textsc{sat}, and random regular $k$-\textsc{sat} for large enough $k\geq k_0$ where $k_0$ is a large non-explicit constant. Establishing the same for small values of $k\geq 3$ remains an important open problem in the study of random \textsc{csp}s.

In this work, we study two closely related models of random \textsc{csp}s, namely the $2$-coloring on random $d$-regular $k$-uniform hypergraphs and the random $d$-regular $k$-\textsc{nae-sat} model. For every $k\geq 3$, we prove that there is an explicit $d_\star(k)$ which gives a satisfiability upper bound for both of the models. Our upper bound $d_\star(k)$ for $k\geq 3$ matches the prediction from statistical physics for the hypergraph $2$-coloring by Dall’Asta, Ramezanpour, Zecchina (2008), thus conjectured to be sharp. Moreover, $d_\star(k)$ coincides with the satisfiability threshold of random regular $k$-\textsc{nae-sat} for large enough $k\geq k_0$ by Ding, Sly, Sun (2014). 

\end{abstract}
\section{Introduction}
In this work, we study the $2$-coloring on random $d$-regular $k$-uniform hypergraphs and the random $d$-regular $k$-\textsc{nae-sat} model for $k\geq 3$. We establish an explicit well-defined upper bound on the satisfiability/colorability threshold that holds for every $k\geq 3$, which is conjectured to be sharp in statistical physics \cite{Dall_Asta_2008} for hypergraph $2$-coloring, and matches the previous rigorous results for random regular $k$-\textsc{nae-sat} model for $k$ large enough \cite{dss14stoc}.

% Our upper bound on the satisfiability threshold $d_\star(k)$ is conjectured to be sharp in statistical physics \cite{Dall_Asta_2008}, and matches the previous rigorous results on the satisfiability threshold of random regular $k$-\textsc{nae-sat} model for $k$ large enough \cite{dss14stoc}.

Given a $k$-uniform hypergraph with $n$ nodes and $m$ hyperedges, where every edge consists of $k$ nodes, a hypergraph $2$-coloring is an assignment of colors from $\{\sf{red}, \sf{blue}\}\equiv \{0,1\}$ to the nodes such that there is no monochromatic hyperedge. If there is such a $2$-coloring, the hypergraph is said to be colorable or satisfiable. It is a typical example of a constraint satisfaction problem (\textsc{csp}) that has been studied extensively in combinatorics and computer science literature \cite{seymour74,Alon88,Achlioptas02,cz12, DYER15, Henning13, HENNING18}.

A $k$-\textsc{nae-sat} problem is another closely related \textsc{csp} studied in computer science \cite{cp12,dss14stoc, ssz16,nss22FOCS, SS23}, which can be viewed as a variant of the infamous $k$-\textsc{sat} problem \cite{Karp72}. A $k$-\textsc{sat} formula is a boolean \textsc{cnf} formula with $n$ variables formed by taking the \textsc{and} of $m$ clauses, which is the \textsc{or} of $k$ variables or their negations. Then, a \textsc{nae-sat} solution $\ux\in \{0,1\}^n$ is an assignment such that $\ux$ and its negation $\neg\ux$ evaluates $\sf{true}$ in the formula. Thus, viewing each clause as an hyperedge, if no variable is negated in every clauses, then a \textsc{nae-sat} solution is equivalent to a hypergraph $2$-coloring.  

A significant direction of research on satisfiability has involved examining the large-system limit of randomly generated problem instances. The study of random constraint satisfaction problems (r\textsc{csp}s) aims to discern typical behaviors and phase transitions in these systems as the number of variables $n$ and the number of constraints $m$ tends to infinity with a fixed ratio $\alpha \equiv \frac{m}{n}$. In this sparse regime, there has been considerable effort into identifying the satisfiability transition, or the critical density, denoted by $\alpha_{\sf sat}$, beyond which solutions cease to exist \cite{ap04, anp05, am06, cv13}.

Many of the sparse r\textsc{csp}s belong to a broad universality class called the one-step-replica-symmetry-breaking (1\textsc{rsb}) class from statistical physics \cite{kmrsz07} (see Chapter 19 of \cite{mm09} for a survey) - including $2$-coloring on random regular $k$-uniform hypergraphs, random regular $k$-\textsc{nae-sat}, and random $k$-\textsc{sat} for $k\geq 3$. The 1\textsc{rsb} class refers to r\textsc{csp} which is predicted to possess a single layer of hierarchy of well-separated clusters, where a cluster roughly refers to a dense region of the solution space. A shared characteristic of these problems is that in a non-trivial regime below $\alpha_{\sf sat}\equiv \alpha_{\sf sat}(k)$, the number of solutions fails to concentrate about its mean due to the clustering effect. This effect thus prevents standard first and second moment
methods from locating the exact transition, presenting a significant mathematical challenge.

Despite such difficulties, breakthroughs were made to successfully locate the satisfiability threshold of the random regular $k$-\textsc{nae-sat} \cite{dss16}, the random $k$-SAT \cite{DSS22}, and random regular $k$-SAT \cite{cp16} for large enough $k\geq k_0$, where $k_0$ is a non-explicit large absolute constant. These works carried out a demanding second
moment method to the number of clusters instead of the number of solutions based on intuitions from statistical physics \cite{mpz02} and previous mathematical works \cite{ap04, COGoing13, cp16}. See Section \ref{subsec:related} for further literature.  

However, for small values of $k\geq 3$, locating the satisfiability threshold for r\textsc{csp}s in the 1\textsc{rsb} class remains an important open problem. Indeed, for all the aforementioned models in 1\textsc{rsb} class, the physicists
conjecture an explicit value $\alpha_{\star}(k)$ for $\alpha_{\sf sat}(k)$, the 1\textsc{rsb} threshold, which is expected to be
correct for all $k\geq 3$ \cite{Mertens06,mpz02,Dall_Asta_2008}. The methods of \cite{dss16,DSS22,cp16} crucially uses the fact that $k$ is large enough for their second moment method to succeed.

In this work, we consider $2$-coloring on random $d$-regular $k$-uniform hypergraphs, where the random hypergraph is generated uniformly at random from the set of $k$-uniform hypergraphs such that every variable participates in exactly $d$ hyperedges. We also consider random $d$-regular \textsc{nae-sat}, where $k$-\textsc{sat} formula is generated uniformly at random with the condition that every variable participates in exactly $d$ clauses. We establish an upper bound $d_\star(k)$ on the satisfiability thresholds for these problems for every $k\geq 3$, which is sharp \cite{dss14stoc} for random regular $k$-\textsc{nae-sat} for large $k\geq k_0$ and conjectured to be sharp \cite{Dall_Asta_2008} for $k\geq 3$ for hypergraph $2$-coloring.

 \begin{thm}\label{thm:main}
 For $k\geq 3$ and $d_{\textnormal{lbd}}(k)\leq d\leq d_{\textnormal{ubd}}(k)$, where $d_{\textnormal{lbd}}(k), d_{\textnormal{ubd}}(k)$ are defined in \eqref{eq:interval:d} below, there exists a unique solution $x\equiv x(k,d)$ to the equation
 \begin{equation}\label{eq:BP}
 d = 1+\left( \log \frac{1-2x}{1-x} \right)/ \log \left( \frac{1-2x^{k-1}}{1-x^{k-1}} \right)\quad\textnormal{on the interval}\quad\frac{1}{2} - \frac{1}{2^k} \le x \le \frac{1}{2}.
 \end{equation}
Define $d_\star(k)$ by the largest zero of the explicit function
\begin{equation}\label{eq:def:Phi}
    \boldsymbol{\prescript{\star}{}{\Phi}}(d) := -\log(1-x) - d(1-k^{-1} - d^{-1}) \log (1-2x^k) + (d-1) \log(1-x^{k-1})\,,
\end{equation}
 where the existence of the root of $\boldsymbol{\prescript{\star}{}{\Phi}}(d)$ is guaranteed in the interval $[d_{\textnormal{lbd}}(k), d_{\textnormal{ubd}}(k)]$.
 
 Then, for $k \ge 3$, and $d>d_\star(k)$, the random $d$-regular $k$-uniform hypergraph is not $2$-colorable with probability tending to one as the graph size $n\to\infty$. Similarly for $k\geq 3$ and $d>d_\star(k)$, then the random $d$-regular $k$-\textsc{nae-sat} instance is not satisfiable with probability tending to one as $n\to\infty$.
 \end{thm}
A matching lower bound was obtained in \cite{dss14stoc} for large enough $k\geq k_0$ in random $d$-regular \textsc{nae-sat} by a demanding second moment method. Our proof is based on an interpolation method from statistical physics \cite{Franz03,guerra03,PT04}. We give a proof outline in Section \ref{subsec:proof}.

We emphasize that for any $k\geq 3$, determining the colorability threshold for $2$-coloring on random $d$-regular $k$-uniform hyprgraphs was previously open, thus Theorem \ref{thm:main} for $2$-coloring is novel even for large $k$. Although it is expected that the colorability threshold for the model matches the satisfiability threshold for random regular $k$-\textsc{nae-sat}, it is highly non-trivial to modify the proof techniques for random regular \textsc{nae-sat} \cite{dss16} to the $2$-coloring model since many of the arguments in \cite{dss16} crucially takes advantage of the randomness of clauses. For example, any $\ux \in \{0,1\}^n$ has the same probability of being a \textsc{nae-sat} solution by the randomness of the clauses while this is obviously not true for the $2$-coloring model. As we see below, even the calculation of the first moment of the solutions is substantially more involved for the $2$-coloring model. Let $Z_{\textsc{nae}}$ be the number of solutions of random $d$-regular $k$-\textsc{nae-sat}, then it is trivial to calculate $\E Z_{\textsc{nae}}$ exactly by taking advantage of the randomness of the clauses:
\begin{equation}\label{eq:first:mo}
\E Z_{\textsc{nae}} = 2^n(1-2^{-k+1})^m= \exp\Big(n\Big(\log 2+\alpha\log\big(1-2^{-k+1}\big)\Big)\Big)=:\exp\big(n\Phi_k(\alpha)\big)\,.
\end{equation}
On the other hand, if we denote $Z_{\textsc{col}}$ by the number of $2$-colorings on random $d$-regular $k$-uniform graphs, then estimating $\E Z_{\textsc{col}}$ is more delicate: we appeal to the idea of exponential tilting from large deviations theory \cite{DZ10} and local central limit theorem \cite{Borokov17} to prove that $\E Z_{\textsc{col}}$ is of the same order as $\exp\big(n\Phi_k(\alpha)\big)$ in Lemma \ref{lem:firstmo:coloring} below. Using the interpolation bound which is simpler than moment calculations, we clarify a simple mechanism (cf. Lemma \ref{lem:symmetry}) behind the identical satisfiability upper bounds for both models.

The solution $x(k,d)$ to the equation \eqref{eq:BP} has a mathematical interpretation. Namely, $2x(k,d)$ is the fraction of the so-called \textit{frozen} variables in the \textit{cluster} model. The solution $x(k,d)$ is called the \textit{Belief Propagation}(\textsc{bp}) fixed point for the cluster model in statistical physics. We emphasize that addressing the uniqueness of the \textsc{bp} fixed point is a well-known major obstacle for many combinatorial optimization and statistical inference problems that exhibit sharp phase transitions (e.g. for spherical perceptron model \cite{shcherbina2003rigorous}; see \cite[Chapter 3]{TalagrandVolI} for a further discussion). We establish the uniqueness of the \textsc{bp} fixed point by showing that the \textit{Belief Propagation recursion} (cf. \eqref{eq:def:var:clause:BP}) is a contraction for $k\geq 3$ and $[d_{\textnormal{lbd}}(k), d_{\textnormal{ubd}}(k)]$, which might be also useful in obtaining a matching lower bound to Theorem \ref{thm:main}.

 \begin{table}[t]
    \begin{center}
        \begin{tabular}{| c | c | c | c | c | c | c | c | c | c | c | c | c | c |}
        \hline
              k & 3 & 4 & 5 & 6 & 7 & 8 & 9 & 10 & 11 & 12 & 13 & 14 & 15   \\
              \hline
            $\lceil d_{\star}(k) \rceil$ & 7 & 20 & 53 & 130 & 307 & 705 & 1592 & 3543 & 7802 & 17028 & 36902 & 79488 & 170340 \\
            \hline 
            $\lceil d_1(k) \rceil$  & 8  & 21 & 54 & 131 & 309 & 708  & 1594 & 3546 & 7804 & 17031 & 36905 & 79491 & 170343\\
            \hline
        \end{tabular}
        \caption{\label{table} A comparison with the upper bound $d_\star(k)$ in Theorem \ref{thm:main} with the first moment threshold $d_1(k):=\frac{k\log 2}{-\log \left(1-2^{-k+1}\right)}$ for small values of $k$. For $3\leq k \leq 10$, the values also appear in Table 1 of \cite{Dall_Asta_2008}.}
    \end{center}
    \end{table}
    
Since $\E Z_{\textsc{nae}}$ and $\E Z_{\textsc{col}}$ are given by $\exp\big(n\Phi_k(\alpha)\big)$ up to a constant (cf. \eqref{eq:first:mo} and Lemma \ref{lem:firstmo:coloring}), the first moment thresholds for both of the models are given by $d_1(k):=\frac{k\log 2}{-\log \left(1-2^{-k+1}\right)}$. In Table \ref{table}, we report $\lceil d_\star(k)\rceil$ and $\lceil d_1(k)\rceil$ for $3\leq k \leq 15$. For every $3\leq k \leq 15$, the upper bound $\lceil d_\star(k) \rceil$ in Theorem \ref{thm:main} improves over the first moment threshold. For large values of $k$, $d_\star(k)$ improves over $d_1(k)$ by $\Omega(k)$ (see \eqref{eq:large:k} below). The quantities $d_{\textnormal{lbd}}(k), $ and $d_{\textnormal{ubd}}(k)$ are defined by
\begin{equation}\label{eq:interval:d}
d_{\textnormal{lbd}}(k)=
\begin{cases}
6.74  & k=3\,,\\
16.7   &k=4\,,\\
(2^{k-1}-2)k \log 2 &k\geq 5\,.
\end{cases}
\quad\quad\quad d_{\textnormal{ubd}}(k)
=
\begin{cases}
7.5  & k=3\,,\\
2^{k-1}k\log 2   &k\ge 4\,.
\end{cases}
\end{equation}

    \begin{remark}
For $d\leq d_{\textnormal{lbd}}(k)$ and large $k\geq k_0$, the second moment method applied to $Z_{\textsc{nae}}$ succeeds in showing the satisfiability for the random $d$-regular $k$-\textsc{nae-sat} model (see \cite[Section 2.1]{dss14stoc}). For $k\in \{3,4\}$, $d_{\textnormal{lbd}}(k)$ must be adjusted to be higher to guarantee that $\boldsymbol{\prescript{\star}{}{\Phi}}(d)$ is well-defined, i.e. there exists a unique solution to \eqref{eq:BP}. The value $d_{\textnormal{ubd}}(k)\equiv 2^{k-1}k\log 2>d_1(k)$ for $ k\geq 4$ is a convenient upper bound for satisfiability. For $k=3$, we take $d_{\textnormal{ubd}}(3)$ to be $7.5>\frac{3\log 2}{-\log(3/4)}=d_1(3)$, which does not change $d_\star(3)$, but is more convenient for the proof.
\end{remark}

Finally, we note that the large $k$ asymptotics of $d_\star(k)$ was proven in \cite[Appendix B]{sly_sun_zhang_2021}:
\begin{equation}\label{eq:large:k}
\alpha_\star(k)\equiv \frac{d_\star(k)}{k}=\left(2^{k-1} -\frac{1}{2}-\frac{1}{4\log 2}\right)\log 2 +o_k(1)\,,
\end{equation}
where $o_k(1)$ denotes an error tending to zero as $k\to\infty$. Since $d_1(k)=(2^{k-1}-1/2)k\log 2 +o_k(1)$, we have that $d_\star(k)\leq d_1(k)-\Omega(k)$.

\subsection{Related work}\label{subsec:related}

Many of the earlier mathematical works on r\textsc{csp}s focused on  determining their satisfiability thresholds and verifying the sharpness of \textsc{sat}-\textsc{unsat} transitions. For models that are known not to exhibit \textsc{rsb}, such goals were established. These models include random 2-\textsc{sat} \cite{cr92,bbckw01},  random 1-\textsc{in}-$k$-\textsc{sat} \cite{acim01}, $k$-\textsc{xor-sat} \cite{dm02,dgmm10,ps16}, and random linear equations \cite{acgm17}. On the other hand, for the models which are predicted to belong to 1\textsc{rsb} class, intensive studies have been conducted to estimate their satisfiability threshold, as shown in \cite{kkks98,ap04,cp16} (random $k$-\textsc{sat}), \cite{am06,cz12,cp12} (random $k$-\textsc{nae-sat}), and \cite{an05,c13,cv13,ceh16} (random graph coloring). 
	
	More recently, the satisfiability thresholds for r\textsc{\textsc{csp}}s that exhibits \textsc{rsb} have been rigorously determined for several models, namely the random regular $k$-\textsc{nae-sat} \cite{dss16},  maximum independent set on $d$-regular graphs \cite{dss16maxis}, random regular $k$-\textsc{sat} \cite{cp16} and  random $k$-\textsc{sat} \cite{DSS22} for large $k$ and $d$. Although determining the location of $q$-colorability threshold for the sparse Erdos Renyi graph is left open, the \textit{condensation threshold} $\alpha_{\sf cond}$ for random graph coloring, where the \textit{free energy} becomes non-analytic, was settled in \cite{bchrv16}. They carried out a technically challenging analysis based on a clever ``planting" technique, where the results were further generalized to other models in \cite{ckpz18}. Similarly,~\cite{bc16} identified the condensation threshold for random regular $k$-\textsc{sat}, where each variable appears $d/2$-times positive and $d/2$-times negative. Further, in the condensation regime $\alpha\in (\alpha_{\sf cond},\alpha_{\sf sat})$, many quantities of interest was established for random regular $k$-\textsc{nae-sat} with large enough $k$, matching the statistical physics prediction. Namely, the number of solutions at exponential scale (free energy) \cite{sly_sun_zhang_2021}, the concentration of the \textit{overlap} \cite{NSS, nss2}, and the local weak limit \cite{SS23} were established. Establishing the same quantities for other models in the condensation regime is left open. 

The closest result to ours in the literature is by Ayre, Coja-Oghlan, and Greenhill \cite{Ayre22}, where they lower bound the chromatic number (or equivalently, upper bound the colorability threshold) of the random regular graph of any degree, which is conjectured to be tight. \cite{Ayre22} also considers the sparse Erdos Renyi graph, which is more complicated since the conjectured chromatic number is defined in terms of a distributional (rather than real-valued) optimization due to the randomness of the local neighborhoods. In this work, we do not consider Erdos Renyi type problems, but we additionally address the question of the uniqueness of the \textsc{bp} fixed point for any $k\geq 3$ (unique solution to the equation \eqref{eq:BP}). As in \cite{Ayre22}, we use an interpolation bound, which gives an upper bound of the satisfiability threshold also for the (non-regular) random $k$-\textsc{nae-sat} model. It would be interesting to address the uniqueness of the \textsc{bp} fixed point for random $k$-\textsc{nae-sat} and random $k$-sat for small $k\geq 3$. We refer to \cite{shcherbina2003rigorous, montanari2019generalization, yu2022ising, gu2023uniqueness} which addresses the uniqueness of \textsc{bp} fixed point for various models.

\subsection{Proof methods}\label{subsec:proof}
We aim to rigorously establish the upper bound the satisfiability threshold predicted by the so-called `1\textsc{rsb} cavity method' from statistical physics \cite{Dall_Asta_2008}. To do so, instead of using moment methods, we use a technique called `interpolation method' from the theory of spin glasses developed by \cite{Franz03, guerra03, PT04}. The interpolation method has been successful in upperbounding the satisfiability threshold for random $k$-\textsc{sat} \cite{dss15ksat} for large $k$, the free energy for random regular $k$-\textsc{nae-sat} \cite{ssz16}, and the colorability threshold for random graphs \cite{Ayre22}. 

We first introduce the notations and mathematical framework that we use throughout the paper. For both the $d$-regular $k$-uniform hypergraphs and the $k$-\textsc{nae-sat} formula, we can represent them as (labelled) $(d,k)$-regular bipartite graph. Let $V=\{v_1,\ldots, v_n\}$ be the set of variables or nodes and $F= \{a_1,\ldots, a_m\}$ be the set of clauses or hyperedges. An edge is formed if the variable or node $v_i$ is included in the clause or hyperedge $a_j$. For an edge $e$, we denote $v(e)$ (resp. $a(e)$) by the variable (resp. clause) adjacent to it.

Denote $G=(V,F,E)$ by the resulting bipartite graph. We denote the neighborhood of $v\in V$ (resp. $a\in F$) by $\delta v:=\{a\in F: (a v)\in E\}$ (resp. $\delta a:=\{v\in V: (a v)\in E\}$). Throughout, we denote $\alpha\equiv \frac{m}{n}=\frac{d}{k}$. For the \textsc{nae-sat} formula, there is an extra label for each edge $e\in E$, namely the \textit{literal} $\tL_e\in \{0,1\}$, which specifies how the variable $v(e)$ participates in the clause $a(e)$. Then, the labelled graph $\GG=(V,F,E, \uL)\equiv (V,F,E, (\tL_e)_{e\in E})$ represents a \textsc{nae-sat} instance. 

\begin{defn}\label{def:model}
Given a \textsc{nae-sat} instance $\GG=(V,F,E, \uL)$, $\ux\in \{0,1\}^V$ is a (\textsc{nae-sat}) \textbf{solution} if
\[
\prod_{a\in F}\varphi((x_{v(e)}\oplus \tL_e)_{e\in \delta a})=1\,,
\]
where for $\underline{z}=(z_i)_{i \leq k} \in \{0,1\}^k$, $\varphi(\underline{z})\equiv\one(z_1=\ldots=z_k)$, and $\oplus$ denotes addition mod $2$. Given a graph $G=(V,F,E)$, $\ux \in \{0,1\}^V$ is a (hypergraph $2$-) \textbf{coloring} if $\ux$ is a \textsc{nae-sat} solution on $G$ with literals identically zero $(G,\underline{0})$.
\end{defn}
The configuration model can be described as follows. Add $d$ (resp. $k$) half-edges adjacent to each variable (resp. each clause) so that there are total $nd=mk$ number of half-edges adjacent to variables (resp. clauses). Thus, $E$ can be regarded as the perfect matching  between to the set of half-edges adjacent to variables to those adjacent to clauses, and hence a permutation in $S_{nd}$. Then, the configuration model $\bG=(V,F,\bE)$ is defined by taking $\bE \sim \textnormal{Unif}(S_{nd})$. For a random $d$-regular $k$-\textsc{nae-sat} instance $\bGG = (\bG, \buL)$, we take the literals $\buL\equiv (\bL_e)_{e\in E} \stackrel{i.i.d.}{\sim} \textnormal{Unif}(\{0,1\})$.

Note that the configuration model $\bG$ may induce multi-edges. However, if we denote $\mathscr{S}$ to be the event that $\bG$ is simple, then it is well-known that $\P(\bG\in \mathscr{S})=\Omega(1)$ (see e.g. Chapter 9 of \cite{jlrrg}). Thus, the configuration model is mutually contiguous with respect to the uniform distribution among all $(d,k)$-regular graphs, so to prove Theorem \ref{thm:main}, it suffices to work with the configuration model. 

In order to use the interpolation method, we consider the \textit{positive temperature} analogs of the $2$-coloring or the \textsc{nae-sat} model, which have more desirable properties due to the softness of the constraints - e.g. the concentration of the free energy as seen in Lemma \ref{lem:free:energy:concentration} below. We introduce notations that allow us to set up the positive temperature models. Let $S$ be a finite set and $\underline{b}\equiv (b_s)_{s\in S}$ be a vector with $b_s\geq 0$. Also, let $\mathcal{X}$ be a finite set encoding the spins and denote $\mathfrak{F}(\mathcal{X})$ by the set of functions $\mathcal{X}\to \R_{\geq 0}$. Let $f:S\to \mathfrak{F}(\mathcal{X})$ be a random function, i.e. $f(\cdot;s)\in \mathfrak{F}(\mathcal{X})$ is random for $s\in S$, and $f_1,\ldots, f_k$ be i.i.d. copies of $f$. Then, define the random function $\theta: \mathcal{X}^k \to \R$ by letting for $\ux=(x_1,\ldots, x_k)\in \mathcal{X}^k$,
\begin{equation}\label{eq:pre:theta}
\theta(\ux)=\sum_{s\in S}b_s \prod_{i=1}^{k}f_j(x_j;s)\,.
\end{equation}

We assume that there exists a constant $\eps\in (0,1)$ such that for any $\ux\in \mathcal{X}^k$,
\begin{equation}\label{eq:theta:finite:temp}
\eps\leq 1-\theta(\ux)\leq \eps^{-1} \quad\textnormal{almost surely.} 
\end{equation}
On a $(d,k)$-regular bipartite graph $G=(V,F,E)$, let $(\theta_a)_{a\in F}$ be i.i.d. copies of the random function $\theta$, and define the (random) Gibbs measure on $\mathcal{X}^V$ by
\[
\mu_{G}(\underline{x}) \equiv \frac{1}{Z(G)} \prod_{a \in F} \Big(1- \theta_{a} (\ux_{\delta a})\Big)\,,
\]
 where $Z(G)$ is the normalizing constant explicitly given by 
\begin{equation}\label{eq:partition}
Z(G)\equiv \sum_{\ux \in \mathcal{X}^{V}} \prod_{a\in F}\Big(1- \theta_{a} (\ux_{\delta a})\Big)\,.
\end{equation}
We note that the condition \eqref{eq:theta:finite:temp} on $\theta$ guarantees that the Gibbs measure $\mu_G$ is `finite temperature'. In particular, if we define the free energy
\begin{equation}\label{eq:def:free:energy}
F_n\equiv \frac{1}{n}\E \log Z(\bG)\,,
\end{equation}
where $\bG$ is drawn from the configuration model and $\E$ above is over the randomness of $\bG$ and randomness of $(\theta_a)_{a\in F}$, we have the following concentration of the free energy.

\begin{lemma}\label{lem:free:energy:concentration}
Assume that $\theta$ satisfies \eqref{eq:theta:finite:temp} with some constant $\eps \in (0,1)$. Then, for any $\delta >0$, there exists a constant which only depends on $\eps,\delta>0$ such that
\[
\P\left(\bigg|\frac{1}{n}\log Z(\bG)-F_n\bigg| \geq \delta \right)\leq e^{-cn}\,.
\]
\end{lemma}
The concentration of free energy in Lemma \ref{lem:free:energy:concentration} is standard in literature \cite{bchrv16,CP19,Ayre22}, and we provide the proof in Section \ref{sec:2} for completeness.

\begin{defn}\label{def:pos}
(Positive temperature models) For $\beta>0$, called the inverse temperature, the positive temperature \textsc{nae-sat} model $\theta_{\textsc{nae}}(\cdot)\equiv \theta_{\textsc{nae}}(\cdot\,;\,\beta)$ is defined as follows. Let $\uL \equiv (\tL_i)_{i \le k}\stackrel{i.i.d.}{\sim}\textnormal{Unif}(\{0,1\})$ be a sequence of i.i.d. Bernoulli(1/2) random variables. Then for $\underline{x}=(x_i)_{i\leq k} \in \{0,1\}^{k}$, define 
\begin{equation}\label{eq:theta:nae}
    \theta_{\textsc{nae}}(\underline{x}) \equiv \theta_{\textsc{nae}}(\underline{x};\beta)\equiv (1-e^{-\beta} )\cdot \left(\prod_{i=1}^{k} (\tL_i \oplus x_i) + \prod_{i=1}^{k} (\tL_i \oplus x_i \oplus 1)
 \right)\,.
\end{equation}
That is, in the general form \eqref{eq:pre:theta}, we take $S=\mathcal{X}=\{0,1\}$, $b_i\equiv 1-e^{-\beta}$, and $f(x;0)\equiv 1-f(x;1)\equiv \one(x\oplus \tL)$ for $\tL\sim \textnormal{Unif}(\{0,1\})$. Moreover, the positive temperature hypergraph $2$-coloring model $\theta_{\textsc{col}}(\cdot)\equiv \theta_{\textsc{col}}(\cdot\,;\,\beta)$ is defined by taking $\tL_i\equiv 0$ above:
 \begin{equation}\label{eq:theta:col}
\theta_{\textsc{col}}(\underline{x}) \equiv \theta_{\textsc{col}}(\underline{x};\beta)\equiv (1-e^{-\beta} )\cdot\sum_{s\in \{0,1\}}\prod_{i=1}^{k}\one(x_i=s)\,,
 \end{equation}
 which is taking $f(x;s)=\one(x=s)$ in \eqref{eq:pre:theta}.
 \end{defn}
We note that formally taking $\beta = \infty$ and $\theta=\theta_{\textsc{col}}(\underline{x};\beta)$, the corresponding partition function $Z(G)$ equals the number of $2$-coloring on $G$. A similar statement holds for the \textsc{nae-sat} model.

By constructing a certain sequential coupling of the given factor graph $(\bG,(\theta)_{a\in F})$ to a set of disjoint trees so that the free energy is monotone at every step, the interpolation method \cite{Franz03,guerra03,PT04} gives an upper bound on the free energy $F_n$ as follows: for $\zeta \in \mathscr{P}\big(\mathscr{P}(\mathscr{P}(\mathcal{X}))\big)$, where $\mathscr{P}(A)$ denotes the set of probability measures on $A$, and $\lambda \in (0,1)$, there exists an explicit functional $\mathcal{P}(\zeta, \lambda)\equiv \mathcal{P}_{d,k,\theta}(\zeta, \lambda)$ such that we have $F_n\leq \inf_{\zeta,\lambda}\mathcal{P}(\zeta,\lambda)+o_n(1)$. By taking advantage of the interpolation method applied to positive temperature models in Definition \ref{def:pos} and the concentration of the free energy in Lemma \ref{lem:free:energy:concentration}, we prove the proposition below in Section \ref{sec:2}.

\begin{prop}\label{prop:interpolation}
For a given $k\ge 3$ and $d$, suppose that there is a solution $x\in [1/2-1/2^k,1/2]$ to the \textsc{bp} equation \eqref{eq:BP}. Further, suppose that $\boldsymbol{\prescript{\star}{}{\Phi}}(d)$ in \eqref{eq:def:Phi} defined with such $x$ satisfies $\boldsymbol{\prescript{\star}{}{\Phi}}(d)<0$. Then, with probability tending to one, no \textsc{nae-sat} solution exists on $\bGG$. Also, with probability tending to one, no $2$-coloring exists on $\bG$.
\end{prop}
Moreover, we show that $d_\star(k)$ in Theorem \ref{thm:main} is well-defined and that the assumptions of Proposition \ref{prop:interpolation} are meaningful. Note that the \textsc{bp} equation \eqref{eq:BP} is equivalent to $\Psi_d(x)=x$, where $\Psi_d \equiv \Psi_{k,d}:[0,1]\to [0,1]$ is defined by $ \Psi_d \equiv \dot{\Psi}\circ \hat{\Psi}$ with 
\begin{equation}\label{eq:def:var:clause:BP}
\dot{\Psi}(x)\equiv  \dot{\Psi}_d(x) \equiv \frac{1-x^{d-1}}{2-x^{d-1}}\,,\quad\quad \hat{\Psi}(x)\equiv \hat{\Psi}_k(x)\equiv \frac{1-2x^{k-1}}{1-x^{k-1}}\,.
\end{equation}
The function $\dot{\Psi}(\cdot)$ is \textit{variable} \textsc{bp} \textit{recursion} and $\hat{\Psi}(\cdot)$ is \textit{clause} \textsc{bp} \textit{recursion} (see \cite[Section 3.1]{dss16} for the motivation).
\begin{prop}\label{prop:bp}
For $k\geq 3$ and $d\in [d_{\textnormal{lbd}}(k),d_{\textnormal{ubd}}(k)]$, there exists a unique root to $\Psi_{d}(x)\equiv (\dot{\Psi}\circ \hat{\Psi})(x)=x$ in the interval $x\in [1/2-1/2^k,1/2]$. Thus, $\boldsymbol{\prescript{\star}{}\Phi}(d)$ in equation \eqref{eq:def:Phi} is well-defined. Furthermore, $d\to \boldsymbol{\prescript{\star}{}\Phi}(d)$ is continuous in the interval $d\in [d_{\textnormal{lbd}}(k),d_{\textnormal{ubd}}(k)]$ with $\boldsymbol{\prescript{\star}{}\Phi}(d_{\textnormal{lbd}}(k))>0$ and $\boldsymbol{\prescript{\star}{}\Phi}(d_{\textnormal{ubd}}(k))<0$.
\end{prop}
The proof of Proposition \ref{prop:bp} is given in Section \ref{sec:3} for $k\geq 4$ and in Section \ref{sec:4} for $k=3$, which requires extra numerical estimates. Finally, we show that the first moment $\E Z_{\textsc{col}}$ of the number of $2$-colorings on random $d$-regular $k$-uniform hypergraphs is the same with $\E Z_{\textsc{nae}}$ up to a constant.
\begin{lemma}\label{lem:firstmo:coloring}
For $k\geq 3$, there exist constants $C_{k,d,i}$ for $i=1,2,$, which only depends on $k,d$ such that $\E Z_{\textsc{col}}/ \E Z_{\textsc{nae}}\in [C_{k,d,1}, C_{k,d,2}]$
\end{lemma}

\begin{proof}[Proof of Theorem \ref{thm:main}] By Proposition \ref{prop:bp}, the function $\boldsymbol{\prescript{\star}{}\Phi}(d)$ is well-defined and has a root in the interval $ [d_{\textnormal{lbd}}(k),d_{\textnormal{ubd}}(k)]$. Moreover, since $\boldsymbol{\prescript{\star}{}\Phi}(d_{\textnormal{ubd}}(k))<0$ holds and $\boldsymbol{\prescript{\star}{}\Phi}(\cdot)$ is continuous, we have $\boldsymbol{\prescript{\star}{}\Phi}(d)<0$ for $d\in (d_\star(k), d_{\textnormal{ubd}}(k)]$. Hence, Proposition \ref{prop:interpolation} shows that if $d\in (d_\star(k), d_{\textnormal{ubd}}(k)]$, then the $2$-coloring of random $d$-regular $k$-uniform hypergraph and random $d$-regular $k$-\textsc{nae-sat} is not satisfiable, both with probability tending to one as $n\to\infty$. Further, since $\E Z_{\textsc{col}}\asymp_{k,d} \E Z_{\textsc{nae}}=\exp\big(n\big(\log 2+\alpha\log\big(1-2^{-k+1}\big)\big)\big)$ by Lemma~\ref{lem:firstmo:coloring} and $\log 2+\alpha\log\big(1-2^{-k+1}\big)<0$ holds for $d>d_{\textnormal{ubd}}(k)$, the same is true for $d>d_{\textnormal{ubd}}(k)$ by Markov's inequality.
\end{proof}

\section{Satisfiability upper bound by interpolation}
\label{sec:2}
In this section, we prove Lemma \ref{lem:free:energy:concentration}, Proposition \ref{prop:interpolation}, and Lemma \ref{lem:firstmo:coloring}. We prove Proposition \ref{prop:interpolation} in Section \ref{subsec:interpolation} based on the interpolation bound from statistical physics \cite{Franz03, guerra03}. In Section \ref{subsec:lemma}, we prove Lemma \ref{lem:free:energy:concentration} based on Azuma Hoeffding's inequality applied to the Doob martingale with respect to clause revealing filtration. In Section \ref{subsec:firstmo}, we prove Lemma \ref{lem:firstmo:coloring} based on the local central limit theorem.

\subsection{Proof of Proposition \ref{prop:interpolation}}
\label{subsec:interpolation}
Throughout, we assume that we are given $k\geq 3$ and $d$ such that there is a solution $x\in [1/2-1/2^k,1/2]$ to the equation \eqref{eq:BP}. We use the following \textit{one-step-replica-symmetry-breaking bound} proven in \cite[Theorem E.3]{sly_sun_zhang_2021} for random regular graphs, which is the analog of \cite[Theorem 3]{PT04} for Erdos Renyi graphs.

\begin{thm}\label{thm:ssz}(Theorem E.3 in \cite{sly_sun_zhang_2021}) Let $\mathcal{X}$ and $S$ be finite sets and consider the partition function $Z(G)$ (cf. Eq.~\eqref{eq:partition}), where $\theta$ in \eqref{eq:pre:theta} satisfies the condition \eqref{eq:theta:finite:temp} for some $\eps>0$ and $b_s\geq 0$ holds for $s\in S$. Let $\mathcal{M}_0\equiv \mathscr{P}(\mathcal{X})$ be the space of probability measures over $\mathcal{X}$, $\mathcal{M}_1\equiv \mathscr{P}(\mathcal{M}_0)$ be the space of probability measures over $\mathcal{M}_0$, and $\mathcal{M}_2 \equiv \mathscr{P}(\mathcal{M}_1)$ be the space of probability measures over $\mathcal{M}_1$. For $\zeta \in \mathcal{M}_2$, let $\underline{\eta} = (\eta_{a,j})_{a \ge 0, j \ge 0}$ be an array of i.i.d. samples from $\zeta$. For each index $(a, j)$ let $\rho_{a, j}\in \mathscr{P}(\mathcal{X})$ be a conditionally independent sample from $\eta_{a, j}$, and denote $\underline{\rho} = (\rho_{a,j})_{a \ge 0, j \ge 0}.$ For $x \in \mathcal{X}$ define random variables 
\begin{equation*}
\begin{split}
    u_a(x) \equiv \sum_{\underline{x} \in \mathcal{X}^k} \one \{x_1 = x \} \big(1 - \theta_a (\underline{x})\big) \prod_{j=2}^{k} \rho_{a,j} (x_j)\,,\quad\quad u_a \equiv \sum_{\underline{x} \in \mathcal{X}^k} \big(1 - \theta_a (\underline{x})\big) \prod_{j=1}^{k} \rho_{a,j} (x_j)\,,
\end{split}
\end{equation*}
where we recall that $(\theta_a)_{a\geq 0}$ are i.i.d. copies of the random function $\theta$. For any $\lambda \in (0,1)$ and any $\zeta \in \mathcal{M}_2$,

\begin{equation}\label{eq:1rsb}
\begin{split}
&F_n \le \mathcal{P}(\zeta, \lambda)+O_{\eps}(n^{-1/3})\,,\quad\textnormal{where}\\
&\mathcal{P}(\zeta, \lambda)\equiv \mathcal{P}_{\theta}(\zeta,\lambda):=\lambda^{-1} \mathbb{E} \log \mathbb{E^\prime} \bigg[ \Big( \sum_{x \in \mathcal{X}} \prod_{a=1}^{d} u_a(x) \Big)^{\lambda} \bigg] - (k-1)\alpha \lambda^{-1} \mathbb{E} \log \mathbb{E^\prime} \Big[ \left(u_0 \right)^{\lambda} \Big]\,.
\end{split}
\end{equation}
Here, $F_n$ is the free energy for the configuration model defined in \eqref{eq:def:free:energy}, $\mathbb{E}^\prime$ denotes the expectation over $\underline{\rho}$ conditioned on all else, and $\mathbb{E}$ denotes the overall expectation.
\end{thm}
\begin{remark}
\cite[Theorem E.3]{sly_sun_zhang_2021} is stated more general than Theorem \ref{thm:ssz} by considering independent \textit{external field} $\{h_v\}_{v\in V}$ and random $(b_s)_{s\in S}$. For our purposes, it suffices to consider non-random $b_s\geq 0$ and $h_v\equiv 1$.
\end{remark}
We use Theorem \ref{thm:ssz} for the positive temperature models in Definition \ref{def:pos}. Note that $\theta_{\textsc{nae}}(\cdot\,;\,\beta)$ and $\theta_{\textsc{col}}(\cdot\,;\,\beta)$ satisfies the condition \eqref{eq:theta:finite:temp} with $\eps=e^{-\beta}$. Furthermore, in the bound \eqref{eq:1rsb}, we take $\lambda = \beta^{-1/2}$ and $\zeta\equiv \zeta_{k,d,\beta} \in \mathscr{P}\Big(\mathscr{P}\big(\mathscr{P}(\{0,1\})\big)\Big)$ given by a point mass at $\eta_{k,d,\beta}$:
\begin{equation}\label{eq:def:zeta}
    \zeta_{k,d,\beta} \equiv \delta_{\eta_{k,d,\beta}}\,,
\end{equation}
where $\eta_{k,d,\beta}\in \mathscr{P}(\mathscr{P}(\{0,1\}))$ is defined as follows. Identify $\mathscr{P}(\{0,1\})$ with $[0,1]$ by the map 
\[\rho\in \mathscr{P}(\{0,1\})\leftrightarrow \rho(1)\in [0,1]\,.
\]
Thus, viewing $\eta\equiv \eta_{k,d,\beta} \in \mathscr{P}([0,1])$, define
\begin{equation}\label{eq:def:eta}
    \eta \left(\frac{e^{\beta}}{e^{\beta} + e^{-\beta}} \right)=\eta \left(\frac{e^{-\beta}}{e^{\beta} + e^{-\beta}} \right)=x\,, \quad \eta\left(\frac{1}{2} \right) = 1-2x\,,
\end{equation}
where $x\equiv x(k,d)$ is the \textsc{bp} fixed point, i.e. the solution to the equation \eqref{eq:BP}. Such choice $\zeta_{k,d,\beta}$ is motivated from physics \cite{kmrsz07} and previous mathematical works \cite[Section 3]{dss16} and \cite[Section 4]{DSS22}.

Before proceeding further, we show that if $\zeta$ is given as in \eqref{eq:def:zeta}, \eqref{eq:def:eta}, then $\mathcal{P}(\zeta, \lambda)$ does not depend on literals. More precisely, suppose that $\zeta= \delta_{\eta_0}$, where $\eta_0 \in \mathscr{P}([0,1])$ is such that $\eta_0(\mathrm{d}x)=\eta_0(\mathrm{d}(1-x))$, i.e. $\rho \stackrel{d}{=}1-\rho$ holds for $\rho \sim \eta_0$. For a fixed $\uL=(\tL_i)_{i\leq k}\in \{0,1\}^k$, let
\begin{equation*}
\theta_{\uL}(\ux)= (1-e^{-\beta} )\cdot \left(\prod_{i=1}^{k} (\tL_i \oplus x_i) + \prod_{i=1}^{k} (\tL_i \oplus x_i \oplus 1)\right)\,.
\end{equation*}
With abuse of notation, for $x\in \{0,1\}$ and independent samples $\rho_{a,j}\in \mathscr{P}(\{0,1\})$ from $\eta_{0}$, let
\begin{equation*}
 u_{a,\uL}(x) \equiv \sum_{\underline{x} \in \{0,1\}^k} \one \{x_1 = x \} \big(1 - \theta_{\uL} (\underline{x})\big) \prod_{j=2}^{k} \rho_{a,j} (x_j)\,,\quad\quad u_{\uL}\equiv \sum_{\underline{x} \in \{0,1\}^k} \big(1 - \theta_{\uL} (\underline{x})\big) \prod_{j=1}^{k} \rho_{a,j} (x_j)\,,
\end{equation*}
where we consider $\uL\in \{0,1\}^k$ to be fixed. Then, for a given sequence of literals $\uL_a \in \{0,1\}^k$ for $0\leq a \leq d$, let
\begin{equation}\label{eq:parisi:literals}
\mathcal{P}\big(\delta_{\eta_0}, \lambda; (\uL_a)_{0\leq a \leq d} \big):=\lambda^{-1} \log \mathbb{E^\prime} \Big( \sum_{x \in \{0,1\}} \prod_{a=1}^{d} u_{a,\uL_a}(x) \Big)^{\lambda}  - (k-1)\alpha \lambda^{-1} \mathbb{E} \log \mathbb{E^\prime}  \left(u_{\uL_0} \right)^{\lambda}\,,
\end{equation}
where $\E^\prime$ is the expectation with respect to the independent samples $\rho_{a,j}\in \mathscr{P}(\{0,1\})$ from $\eta_{0}$. Note that if $\uL_a \stackrel{i.i.d}{\sim} \textnormal{Unif}(\{0,1\}^k)$, then $\mathcal{P}_{\theta_{\textsc{nae}}}(\delta_{\eta_0}, \lambda)=\E_{\uL} \mathcal{P}\big(\delta_{\eta_0}, \lambda; (\uL_a)_{0\leq a \leq d} \big)$ holds, and if $\uL_a \equiv \underline{0}$ for $0\leq a \leq d$, then $\mathcal{P}_{\theta_{\textsc{col}}}(\delta_{\eta_0}, \lambda)=\mathcal{P}\big(\delta_{\eta_0}, \lambda; \underline{0}\big)$ holds. The following lemma then clarifies the mechanism behind the identical satisfiability upper bound in Theorem \ref{thm:main}.
\begin{lemma}\label{lem:symmetry}
Consider $\zeta=\delta_{\eta_0}$ for some $\eta_0 \in \mathscr{P}([0,1])$ such that $\eta_0(\mathrm{d}x)=\eta_0(\mathrm{d}(1-x))$. Then, for any literals $\uL_a \in \{0,1\}^k$ for $0\leq a \leq d$, the value $\mathcal{P}\big(\delta_{\eta_0}, \lambda; (\uL_a)_{0\leq a \leq d} \big)$ does not depend on $(\uL_a)_{0\leq a \leq d}$. Thus, $\mathcal{P}_{\theta_{\textsc{nae}}}(\delta_{\eta_0}, \lambda)= \mathcal{P}_{\theta_{\textsc{col}}}(\delta_{\eta_0}, \lambda)$ holds.
\end{lemma}
\begin{proof}
For fixed $\uL_a\in \{0,1\}^k$ for $0\leq a \leq d$, note that the vectors $\big(u_{a,\uL_a}(0),u_{a,\uL_a}(1)\big)$ are independent for $0\leq a \leq d$. Thus, it suffices to show that for given $\uL, \uL^\prime \in \{0,1\}^k$ and $1\leq a \leq d$, 
\begin{equation}\label{eq:lem:symmetry:goal}
u_{\uL}\stackrel{d}{=}u_{\uL^\prime}\quad\textnormal{and}\quad \big(u_{a, \uL}(0),u_{a, \uL}(1)\big)\stackrel{d}{=}\big(u_{a, \uL^\prime}(0),u_{a, \uL^\prime}(1)\big)\,.
\end{equation}
To this end, let $ \uL^\prime= \underline{0}$ and we first prove that $u_{\uL}\stackrel{d}{=}u_{\underline{0}}$ holds. Since $\theta_{\uL}(\ux)=\theta_{\underline{0}}(\ux\oplus \uL)$,
\begin{equation*}
    u_{\uL} \equiv \sum_{\underline{x} \in \{0,1\}^k} \big(1 - \theta_{\uL} (\underline{x})\big)\prod_{j=1}^{k} \rho_{0,j} (x_j)= \sum_{\underline{x} \in \{0,1\}^k} \big(1 - \theta_{\underline{0}} (\underline{x})\big) \prod_{j=1}^{k} \rho_{0,j} (x_j \oplus \tL_j)\,. 
\end{equation*}
Note that since $(\rho_{0,j})_{1\leq j \leq k}$ are i.i.d. samples from $\eta_0$ and $\eta_0(\mathrm{d} x)=\eta_0(\mathrm{d}(1-x))$ holds, the sequence $\big(\rho_{0,j}(\cdot\oplus \tL_j)\big)_{1\leq j \leq k}$ are also i.i.d. from $\eta_0$. Hence, the equation above shows that $u_{\uL}\stackrel{d}{=}u_{\underline{0}}$ holds.

Next, we prove that $\big(u_{a, \uL}(0),u_{a, \uL}(1)\big)\stackrel{d}{=}\big(u_{a, \underline{0}}(0),u_{a, \underline{0}}(1)\big)$ holds. Without loss of generality, let $a=1$. Again since $\theta_{\uL}(\ux)=\theta_{\underline{0}}(\ux\oplus \uL)$,
\begin{align*}
    u_{1,\uL}(x) &\equiv \sum_{\underline{x} \in \{0,1\}^k} \one \{x_1 = x \} \big(1 - \theta_{\uL} (\underline{x})\big) \prod_{j=2}^{k} \rho_{1,j} (x_j)= \sum_{\underline{x} \in \{0,1\}^k} \one \{x_1 \oplus \tL_1 = x \} \big(1 - \theta_{\underline{0}} (\underline{x})\big) \prod_{j=2}^{k} \rho_{1,j} (x_j \oplus \tL_j)
\end{align*}
Now, observe that $\theta_{\underline{0}}(\cdot)$ is invariant under global flip, i.e. $\theta_{\underline{0}}(x)=\theta_{\underline{0}}(x\oplus 1)$. Thus, it follows that
\begin{align*}
    u_{1,\uL}(x) &= \sum_{\underline{x} \in \{0,1\}^k} \one  \{x_1 = x \} [1 - \theta_{\underline{0}} (\underline{x})] \prod_{j=2}^{k} \rho_{1,j} (x_j \oplus \tL_1 \oplus \tL_j)\,.
\end{align*}
By the same reasons as above, $\big(\rho_{1,j}(\cdot\oplus \tL_1\oplus \tL_j)\big)_{2\leq j \leq k}$ have the same distribution as $\big(\rho_{1,j}\big)_{2\leq j \leq k}$, which are i.i.d. from $\eta_0$. Thus, we have that $\big(u_{1, \uL}(0),u_{1, \uL}(1)\big)\stackrel{d}{=}\big(u_{1, \underline{0}}(0),u_{1, \underline{0}}(1)\big)$. Therefore, \eqref{eq:lem:symmetry:goal} holds, which concludes the proof.
\end{proof}
The following lemma relates $\mathcal{P}_{\theta_{\textsc{col}}}(\zeta_{k,d,\beta},\beta^{-1/2})=\mathcal{P}_{\theta_{\textsc{nae}}}(\zeta_{k,d,\beta},\beta^{-1/2}),$ and $\boldsymbol{\prescript{\star}{}{\Phi}}(d)$, which plays a crucial role in proving Proposition \ref{prop:interpolation}. Recall the definition of $\zeta_{k,d,\beta}$ in \eqref{eq:def:zeta} and \eqref{eq:def:eta}.
\begin{lemma}\label{lem:crucial}
$\mathcal{P}_{\theta_{\textsc{col}}}(\zeta_{k,d,\beta},\beta^{-1/2})\leq C+\beta^{1/2} \cdot \boldsymbol{\prescript{\star}{}{\Phi}}(d)$ holds for some constant $C\in \R$, which does not depend on $\beta>0$.
\end{lemma}
\begin{proof}
Throughout, let $(\rho_{a,j})_{a\geq 0, j\geq 0}$ denote i.i.d. samples from $\eta_{k,d,\beta}$ defined in \eqref{eq:def:eta}, and let $\E^{\prime}$ (resp. $\P^\prime$) denote the expectation (resp. probability) with respect to $(\rho_{a,j})_{a\geq 0, j\geq 0}$. Also, we use the generic notation $C$ by a constant that does not depend on $\beta>0$. Note that since $\theta_{\textsc{col}}$ and $\eta_{k,d,\beta}$ are non-random, the outer expectation $\E$ in the definition of $\mathcal{P}(\zeta, \lambda)$ in \eqref{eq:1rsb} is redundant. 

First, we bound the second term of the definition of $\mathcal{P}_{\theta_{\textsc{col}}}(\zeta_{k,d,\beta},\beta^{-1/2})$ in \eqref{eq:1rsb}:
\begin{align*}
(k-1)\alpha \beta^{1/2}  \log \mathbb{E^\prime} \Big[ \left(u_0 \right)^{\beta^{-1/2}} \Big]
    &= (k-1)\alpha \beta^{1/2}\log \mathbb{E^\prime} \Bigg[ \bigg(1 - (1-e^{-\beta})\bigg( \prod_{j=1}^{k} \rho_{0,j} (0)+\prod_{j=1}^{k} \rho_{0,j} (1) \bigg) \bigg)^{\beta^{-1/2}} \Bigg]
\end{align*}
Note that the expectation inside the log in the right hand side above is bounded below by
\[
2^{-\beta^{-1/2}}\cdot \mathbb{P^\prime} \bigg( 1 - (1-e^{-\beta}) \bigg( \prod_{j=1}^{k} \rho_{0,j} (0)+\prod_{j=1}^{k} \rho_{0,j} (1) \bigg) \ge \frac{1}{2} \bigg)= 2^{-\beta^{-1/2}}(1-2x^k)\,,
\]
where $x$ is the solution to the \textsc{bp} equation \eqref{eq:BP} and the equality holds for large enough $\beta\geq \beta_0$ since for large $\beta$ and $k \ge 3,$ $(1-e^{-\beta})\left( \prod_{j=1}^{k} \rho_{0,j} (0)+\prod_{j=1}^{k} \rho_{0,j}(1)\right) \ge \frac{1}{2}$ holds if and only if either $\rho_{0,j}(1) = \frac{e^{\beta}}{e^{\beta}+e^{-\beta}}$ holds for all $1\leq j \leq k$, or $\rho_{0,j}(1) = \frac{e^{-\beta}}{e^{\beta}+e^{-\beta}}$ holds for all $1\leq j \leq k$. Thus, it follows that
\begin{equation}\label{eq:secondterm:bound}
    -(k-1)\alpha \lambda^{-1} \mathbb{E} \log \mathbb{E^\prime} \left[ \left(u_0 \right)^{\lambda} \right] \le C - \beta^{1/2} (k-1)\alpha \log\big(1-2x^k\big)\,.
\end{equation}
Next, we estimate the first term of the definition of $\mathcal{P}_{\theta_{\textsc{col}}}(\zeta_{k,d,\beta},\beta^{-1/2})$ in \eqref{eq:1rsb}, which equals
\begin{equation}\label{eq:first:term:parisi}
\begin{split}
&\beta^{1/2} \log \mathbb{E^\prime} \bigg[ \Big( \sum_{x \in \{0,1\}} \prod_{a=1}^{d} u_a(x) \Big)^{\beta^{-1/2}} \bigg]\\
&=
 \beta^{1/2} \log \mathbb{E^\prime} \left[ \Bigg( \prod_{a=1}^{d} \bigg(1 - (1-e^{-\beta})\prod_{j=2}^{k}\rho_{a,j}(0) \bigg) +\prod_{a=1}^{d}\bigg(1 - (1-e^{-\beta})\prod_{j=2}^{k}\rho_{a,j}(1) \bigg)\Bigg)^{\beta^{-1/2}} \right]
\end{split}
\end{equation}
We upper bound the expectation inside the log in the above expression by
\[\begin{split}
2^{\beta^{-1/2}} \cdot \mathbb{P^\prime}\big(\mathcal{A}\big) +  \left(3e^{-\beta} \right)^{\beta^{-1/2}}\,,
\end{split}\]
where
\[
\mathcal{A}:=\left\{ \prod_{a=1}^{d} \bigg(1 - (1-e^{-\beta})\prod_{j=2}^{k}\rho_{a,j}(0) \bigg) +\prod_{a=1}^{d}\bigg(1 - (1-e^{-\beta})\prod_{j=2}^{k}\rho_{a,j}(1) \bigg)  \ge 3e^{-\beta} \right\}\,.
\]

Define the events $\mathcal{E}_0$ and $\mathcal{E}_1$ involving $(\rho_{a,j})_{1\leq a\leq d, 2\leq j \leq k}$ as follows.
\begin{itemize}
    \item $\mathcal{E}_0$ is the event such that for any $1\leq a\leq d$, there exists $j \in \{2,\ldots, k\}$ such that $\rho_{a,j}(0) \neq \frac{e^{\beta}}{e^{\beta}+e^{-\beta}}$.
    
    \item $\mathcal{E}_1$ is the event such that for any $1\leq a\leq d$, there exists $j \in \{2,\ldots, k\}$ such that $\rho_{a,j}(1) \neq \frac{e^{\beta}}{e^{\beta}+e^{-\beta}}$.
\end{itemize}
We now claim that for large enough $\beta$, the event $\mathcal{A}$ is included in $\mathcal{E}_0\cup \mathcal{E}_1$. To this end, suppose that the event $(\mathcal{E}_0 \cup \mathcal{E}_1)^{\sf c}= \mathcal{E}_0^{\sf c} \cap \mathcal{E}_1^{\sf c}$ holds. Then, for each $x\in \{0,1\}$,  for some $a\equiv a(x) \in \{1,\ldots, d\}$ such that $\rho_{a,j} (x)=\frac{e^{\beta}}{e^{\beta}+e^{-\beta}}$ holds for all $2 \le j\le k$. Thus, for $x\in \{0,1\}$, we have
    \[\prod_{a=1}^d \bigg(1-(1-e^{-\beta})\prod_{j=2}^k \rho_{a,j} (x)\bigg) \le 1-(1-e^{-\beta})\left( \frac{e^{\beta}}{e^{\beta}+e^{-\beta}} \right)^{k-1}\leq e^{-\beta}+e^{-2\beta}\,,\]
    where the last inequality holds for large enough $\beta\geq \beta_k$. Hence, summing over $x\in \{0,1\}$ gives that the event $\mathcal{A}$ cannot hold, which proves our claim that $\mathcal{A}\subset \mathcal{E}_0\cup \mathcal{E}_1$. Consequently, the term \eqref{eq:first:term:parisi} is bounded above by 
    \begin{align*}
    \beta^{1/2} \log \left(2^{\beta^{-1/2}} \cdot \P^\prime\big(\mathcal{E}_0 \cup \mathcal{E}_1\big)   +  (3e^{-\beta})^{\beta^{-1/2}}\right)\leq \beta^{1/2 }\log\P^\prime\big(\mathcal{E}_0 \cup \mathcal{E}_1\big) +C\,.
    \end{align*}
    Note that $\P^\prime\big(\mathcal{E}_0 \cup \mathcal{E}_1\big)$ can be calculated explicitly by
    \[
    \P^\prime(\mathcal{E}_0\cup \mathcal{E}_1)= 2(1-x^{k-1})^d - (1-2x^{k-1})^d = \frac{(1-x^{k-1})^{d-1} (1-2x^k)}{1-x}\,,
    \]
    where in the final equality, we used the fact that $x$ is the solution to the equation \eqref{eq:BP}. Therefore, we have proven that
    \begin{equation}\label{eq:firstterm:bound}
        \beta^{1/2} \log \mathbb{E^\prime} \bigg[ \Big( \sum_{x \in \{0,1\}} \prod_{a=1}^{d} u_a(x) \Big)^{\beta^{-1/2}} \bigg]\leq C + \beta^{1/2} \left( -\log(1-x) + (d-1) \log(1-x^{k-1}) + \log(1-2x^k) \right)\,.
    \end{equation}
    In conclusion, combining \eqref{eq:secondterm:bound} and \eqref{eq:firstterm:bound}, and recalling the definition of $\boldsymbol{\prescript{\star}{}{\Phi}}(d)$ in \eqref{eq:def:Phi}, we have 
    \begin{align*}
        \mathcal{P}_{\theta_{\textsc{col}}}(\zeta_{k,d,\beta},\beta^{-1/2}) 
        &\le C + \beta^{1/2} \boldsymbol{\prescript{\star}{}{\Phi}}(d)\,,
    \end{align*} 
which concludes the proof.
\end{proof}

\begin{proof}[Proof of Proposition \ref{prop:interpolation}]
Given a \textsc{nae-sat} instance $\bGG$, let $\textsf{SOL}(\bGG)\subset \{0,1\}^V$ denotes the set of \textsc{nae-sat} solutions. Also, let $Z_{\beta,\textsc{nae}}(\bGG)$ denotes the partition function \eqref{eq:partition} for $\theta=\theta_{\textsc{nae}}(\cdot\,;\,\beta)$. Note that if $\ux \in \textsf{SOL}(\bGG)$, then $\theta_{\textsc{nae}}(\ux_{\delta a})=0$ for any $a\in F$, thus we have for any $\beta>0$ that
\begin{equation}\label{eq:tech}
Z_{\beta,\textsc{nae}}(\bGG)\equiv \sum_{\ux\in \{0,1\}^V}\prod_{a\in F}\big(1-\theta_{\textsc{nae}}(\ux_{\delta a};\beta)\big)\geq \left|\textsf{SOL}(\bGG)\right|\,.
\end{equation}
On the other hand, since $\theta_{\textsc{nae}}(\cdot\,;\,\beta)$ satisfies the condition \eqref{eq:theta:finite:temp} with $\eps=e^{-\beta}$, we have by Theorem \ref{thm:ssz} that
\begin{equation*}
\frac{1}{n}\E\Big[\log Z_{\beta,\textsc{nae}}(\bGG)\Big]\leq \mathcal{P}_{\theta_{\textsc{nae}}}(\zeta_{k,d,\beta}, \beta^{-1/2})+o_n(1)=\mathcal{P}_{\theta_{\textsc{col}}}(\zeta_{k,d,\beta}, \beta^{-1/2})+o_n(1)\,,
\end{equation*}
where the last equality is due to Lemma \ref{lem:symmetry}. By Lemma \ref{lem:crucial}, the right hand side is further bounded by
\begin{equation*}
\frac{1}{n}\E\Big[\log Z_{\beta,\textsc{nae}}(\bGG)\Big]\leq \beta^{1/2}\cdot \boldsymbol{\prescript{\star}{}\Phi}(d)+C+o_n(1)\,,
\end{equation*}
for some constant $C$ that does not depend on $n$ nor $\beta$. If $\boldsymbol{\prescript{\star}{}\Phi}(d)<0$, then for large enough $\beta>0$, $\beta^{1/2}\cdot \boldsymbol{\prescript{\star}{}\Phi}(d)+C<-1$ holds, thus $n^{-1}\E\big[\log Z_{\beta,\textsc{nae}}(\bGG)\big]<-1$ holds for large enough $n$. For such $\beta=\beta_0(k,d)>0$, we have by \eqref{eq:tech} and Lemma \ref{lem:free:energy:concentration} that for large enough $n$,
\begin{equation*}
\P\Big(\left|\textsf{SOL}(\bGG)\right|\geq 1\Big)\leq \P\bigg(\bigg|\frac{1}{n}\log Z_{\beta_0,\textsc{nae}}(\bGG)-\frac{1}{n}\E\Big[\log Z_{\beta_0,\textsc{nae}}(\bGG)\Big]\bigg|\geq 1\bigg)\leq e^{-cn}\,,
\end{equation*}
for some constant $c$ that depends only on $\beta_0>0$, which finishes the proof for the \textsc{nae-sat} model. 

Given a configuration model $\bG$, let $Z_{\beta,\textsc{col}}(\bG)$ denote the partition function \eqref{eq:partition} for $\theta=\theta_{\textsc{col}}(\cdot\,;\,\beta)$. Then, by the same reasoning, Theorem \ref{thm:ssz} and Lemma \ref{lem:crucial} shows that if $\boldsymbol{\prescript{\star}{}\Phi}(d)<0$ then $\frac{1}{n}\E\big[\log Z_{\beta,\textsc{col}}(\bG) \big]<-1$ holds for large enough $\beta=\beta_0(k,d)>0$ and $n$ large enough. On the event that there exists a $2$-coloring on $\bG$, $Z_{\beta,\textsc{col}}(\bG)\geq 1$ holds, so Lemma \ref{lem:free:energy:concentration} again concludes the proof.
\end{proof}
\subsection{Proof of Lemma \ref{lem:free:energy:concentration}}\label{subsec:lemma}
Recall that $\bG=(V,F,\bE)$ is generated from the configuration model, where the $\bE$ is drawn uniformly from $S_{nd}$. Thus, $\bE$ has the same law as sequentially drawing random clauses $\ba_1,\ldots, \ba_{m}$ as follows. At times $t\in \{1,\ldots,k\}$, clause $\ba_t$ is drawn by connecting the $k$ adjacent half-edges to previously unmatched half-edges adjacent to variables. For $1\leq t \leq m$, let $\FF_t$ be the $\sigma$-algebra generated by $\ba_1,\ldots, \ba_t$, and $\FF_0\equiv \emptyset$. Denote $M_t\equiv \E\big[\log Z(\bG)\mid \FF_t\big]$ by the associated Doob martingale. Note that if $\bG =(V,F,E)$ and $\bG^\prime =(V,F,E^\prime)$ has the the same set of edges except for those adjacent to two clauses $a_1\neq a_2\in F$, then by our assumption of $\theta$ in \eqref{eq:theta:finite:temp} and the definition of $Z(G)$ in \eqref{eq:partition}, it follows that $\eps^2\leq Z(G)/Z(G^\prime)\leq \eps^{-2}$ holds. Thus, we have for every $t\in \{0,1,\ldots,m-1\}$ that
\begin{equation}\label{eq:bounded:diff}
\Big|M_{t+1}-M_t\Big|\equiv \Big|\E\big[\log Z(\bG)\mid \FF_{t+1}\big]-\E\big[\log Z(\bG)\mid \FF_t\big]\Big|\leq 2 \log\big(1/\eps\big)\,,
\end{equation}
from which Lemma \ref{lem:free:energy:concentration} follows.
\begin{proof}[Proof of Lemma \ref{lem:free:energy:concentration}]
Note that $M_m =\log Z(\bG)$ and $M_0 = \E\big[\log Z(\bG)\big]$ holds and $(M_t)_{0\leq t\leq m}$ is a martingale with bounded difference by \eqref{eq:bounded:diff}. Therefore, the conclusion follows from Azuma Hoeffding's inequality. 
\end{proof}
\subsection{Proof of Lemma \ref{lem:firstmo:coloring}}
\label{subsec:firstmo}
The following notations are convenient for the proof of Lemma \ref{lem:firstmo:coloring}. For non-negative quantities $f=f_{d,k, n}$ and $g=g_{d,k,n}$, we use any of the equivalent notations $f=O_{k,d}(g), g= \Omega_{k,d}(f), f\lesssim_{k,d} g$ and $g \gtrsim_{k,d} f $ to indicate that there exists a constant $C_{k,d}$, which only depends on $k,d$ such that $f\leq C_{k,d}\cdot g$. We drop the subscripts $d$ (resp. $k,d$) if the constant $C_{k,d}$ does not depend on $d$ (resp. $k,d$). When $f\lesssim_{k,d} g$ and $g\lesssim_{k,d} f$, we write $f\asymp_{k,d} g$. Similarly when $f\lesssim g$ and $g\lesssim f$, we write $f \asymp g$.

Note that $\E Z_{\textsc{col}}$ is the sum over $\ux\in \{0,1\}^V$ of the probabilities that $\ux$ is a $2$-coloring on $\bG$. By symmetry, the probability of $\ux\in \{0,1\}^V$ being a $2$-coloring depends only on the number $n\gamma$ of nodes having color $1$, which we denote by $\bp_{\gamma}$. Thus, $\E Z_{\textsc{col}}=\sum_{\gamma} \binom{n}{n\gamma} \bp_{\gamma}$, where the sum is over $\gamma \in (0,1)$ such that $n\gamma \in \mathbb{Z}$. Moreover, we can express $\bp_{\gamma}$ as follows. Let $X_1,\ldots, X_m$ be i.i.d. $\textnormal{Binom}(k,\gamma)$ random variables and denote $\P_{\gamma}$ by the probability with repect to $(X_i)_{i\leq m}$. Then, we have
\begin{equation}\label{eq:local:clt}
\begin{split}
\bp_{\gamma}
&=\P_{\gamma}\Big(X_i\notin \{0,k\}~~~\textnormal{for all}~1\leq i \leq m\, \Big|\,\sum_{i=1}^{m} X_i= km\gamma\Big)\\
&\leq \frac{\P_{\gamma}\big(X_i\notin \{0,k\}~~~\textnormal{for all}~1\leq i \leq m \big)}{\P_{\gamma}\big(\sum_{i=1}^{m} X_i= km\gamma\big)}\lesssim_k \sqrt{m}(1-\gamma^k-(1-\gamma)^k)^m\,,
\end{split}
\end{equation}
where the last inequality is due to a Stirling's approximation. It follows that
\begin{equation}\label{eq:bound:firstmo}
\begin{split}
&\E Z_{\textsc{col}}\leq n^{O(1)}\sum_{\gamma} \exp\Big(n F_{\alpha}(\gamma)\Big)\,,\quad\textnormal{where}\\
&F_{\alpha}(\gamma):= H(\gamma)+\alpha \log\big(1-\gamma^k-(1-\gamma)^k\big)\,.
\end{split}
\end{equation}
Here, $H(\gamma)\equiv -\gamma \log \gamma-(1-\gamma) \log(1-\gamma)$ is the entropy of $\gamma$. Note that $\gamma \to \gamma^k+(1-\gamma)^k$ is uniquely minimized at $\gamma=1/2$. Further, the entropy $H(\gamma)$ is strictly concave and is maximized at $\gamma=1/2$. Thus, $\gamma\to F_{\alpha}(\gamma)$ is uniquely maximized at $\gamma=1/2$ with $\frac{\partial^2 F_{\alpha}}{\partial \gamma^2}(1/2)<0$. Since $\E Z_{\textsc{nae}} = \exp\big(nF_{\alpha}(1/2)\big)$, it follows from \eqref{eq:bound:firstmo} that
\begin{equation}\label{eq:bound:with:poly}
\E Z_{\textsc{col}} \leq n^{O(1)}\exp\big(nF_{\alpha}(1/2)\big)=n^{O(1)}\cdot\E Z_{\textsc{nae}}\,.
\end{equation}
We now show that the polynomial factor $n^{O(1)}$ can actually be removed with a matching lower bound.

First, by \eqref{eq:local:clt} and the fact that $\gamma\to F_{\alpha}(\gamma)$ is uniquely maximized at $\gamma=1/2$ with strictly negative second derivative, the contribution to $\E Z_{\textsc{col}}$ from $\gamma$ such that $|\gamma-1/2|\geq n^{-1/3}$ is negligible:
\begin{equation}\label{eq:gamma:negligible}
\sum_{|\gamma-1/2|\geq n^{-1/3}} \binom{n}{n\gamma}\bp_{\gamma}\lesssim_{k,d} \exp\big(-\Omega_{k,d}\big(n^{1/3}\big)\big)\cdot \E Z_{\textsc{nae}}\,.
\end{equation}
Thus, we focus on the regime $|\gamma-1/2|\leq n^{-1/3}$.
Note that we can calculate $\bp_{\gamma}$ by summing over the empirical distribution $\nu$ of $(X_i)_{i\leq m}$. Consider $\nu \in \mathscr{P}(\{1,\ldots,k-1\})$ and let $p_{\gamma}(j):=\binom{k}{j}\gamma^j (1-\gamma)^{k-j}$. Then, 
\begin{equation*}
\begin{split}
\bp_{\gamma}
=\frac{\sum_{\nu} \one\Big(\sum_j j\nu_j=km\gamma\Big)\binom{m}{m\nu}\prod_j p_{\gamma}(j)^{m\nu_j}}{\P_{\gamma}\big(\sum_{i=1}^{m} X_i= km\gamma\big)}=\frac{\sum_{\nu} \one\Big(\sum_j j\nu_j=km\gamma\Big)e^{-km\gamma\lambda} \binom{m}{m\nu}\prod_j (p_{\gamma}(j)e^{\lambda j})^{m\nu_j}}{\P_{\gamma}\big(\sum_{i=1}^{m} X_i= km\gamma\big)}\,,
\end{split}
\end{equation*}
where $\binom{m}{m\nu}\equiv \frac{m!}{\prod_j (m\nu_j)!}$ and we introduced a lagrange parameter $\lambda\in \R$ in the last equality. Let
\[
\nu_{\gamma,\lambda}(x):=\frac{p_{\gamma}(x)e^{\lambda x}}{\sum_{j=1}^{k-1}p_{\gamma}(j)e^{\lambda j}}~~\textnormal{for}~~1\leq x\leq k-1\,,
\]
and denote $\P_{\gamma,\lambda}$ by the probability with respect to $\wt{X}_1,\ldots, \wt{X}_m\stackrel{i.i.d.}{\sim}\nu_{\gamma,\lambda}$. Then, it follows that
\begin{equation}\label{eq:express:bp}
\bp_{\gamma}= \frac{\P_{\gamma,\lambda}\big(\sum_{i=1}^{m}\wt{X}_i=km\gamma \big)}{\P_{\gamma}\big(\sum_{i=1}^{m} X_i= km\gamma\big)}\exp\big(-m\cdot\Xi(\gamma,\lambda)\big)\,,~~\textnormal{where}~~\Xi(\gamma,\lambda):=k\gamma\lambda-\log\bigg(\sum_{j=1}^{k-1}p_{\gamma}(j)e^{\lambda j}\bigg)\,.
\end{equation}
In order to use the local central limit theorem, we take $\lambda=\lambda(\gamma)$ such that $\E_{\gamma,\lambda}\wt{X}=k\gamma$, where $\wt{X}\sim \nu_{\gamma,\lambda}$. The existence of such $\lambda(\gamma)$ is guaranteed by the lemma below.
\begin{lemma}\label{lem:lagrange}
For large enough $n$ and all $\gamma$ such that $|\gamma-1/2|\leq n^{-1/3}$, there exists a unique $\lambda=\lambda(\gamma)$ such that $\E_{\gamma,\lambda}\wt{X}=k\gamma$ holds. Furthermore, we have $\lambda(1/2)=0$ and $\big|\lambda(\gamma)\big|\lesssim_{k}n^{-1/3}$ holds uniformly over $|\gamma-1/2|\leq n^{-1/3}$.
\end{lemma}
\begin{proof}
Note that we have $\frac{\partial \Xi}{\partial \lambda}(\gamma,\lambda)=k\gamma -\E_{\gamma,\lambda}\wt{X}$ by definition of $\nu_{\gamma,\lambda}$ and $\Xi(\gamma,\lambda)$. Further, we have that 
\[
\frac{\partial \Xi}{\partial \lambda}\Big(\frac{1}{2}\,,\,0\Big)=\frac{k}{2}-\E_{\frac{1}{2},0}\wt{X}=\frac{k}{2}-\E_{\frac{1}{2}}\big[X \,\big|\,X\notin \{0,k\}\big]=0\,,
\]
where $\E_{\frac{1}{2}}$ is with respect to $X\sim \textnormal{Binom}(1/2)$. Since $\lambda \to \log\big(\sum_{j=1}^{k-1}p_{\gamma}(j)e^{\lambda j}\big)$ is strongly convex, we have $\frac{\partial^2\Xi}{\partial \lambda^2}\big(\frac{1}{2},0\big)<0$. Thus, implicit function theorem shows that for $\gamma \in (1/2-\eps,1/2+\eps)$, where $\eps=\eps(k)>0$ depends only on $k$, there exists $\lambda=\lambda(\gamma)$ such that $\frac{\partial \Xi}{\partial \lambda}\big(\gamma,\lambda(\gamma)\big)=0$ holds, and that $\gamma \to \lambda(\gamma)$ is continuously differentiable. Therefore, for large enough $n$ and $\gamma\in (1/2-n^{-1/3},1/2+n^{1/3})$, there exists a unique $\lambda=\lambda(\gamma)$ such that $\E_{\gamma,\lambda(\gamma)}\wt{X}=k\gamma$, and $|\lambda(\gamma)|\lesssim_k n^{-1/3}$ holds uniformly over $\gamma\in (1/2-n^{-1/3},1/2+n^{1/3})$.
\end{proof}
Having Lemma \ref{lem:lagrange} in hand, we prove Lemma \ref{lem:firstmo:coloring} by appealing to the local central limit theorem.
\begin{proof}[Proof of Lemma \ref{lem:firstmo:coloring}]
The contribution to $\E Z_{\textsc{col}}$ from $\gamma$ such that $|\gamma-1/2|\geq n^{-1/3}$ is negligible by \eqref{eq:gamma:negligible}, thus we consider $\gamma$ such that $|\gamma-1/2|\leq n^{-1/3}$ holds. To this end, we take $\lambda=\lambda(\gamma)$ from Lemma \ref{lem:lagrange} in equation \eqref{eq:express:bp}. Then, by the local central limit theorem \cite{Borokov17},
\begin{equation}\label{eq:bp:estimate}
\bp_{\gamma}\asymp \bigg(\frac{\Var_{\gamma}\big(X\big)}{\Var_{\gamma,\lambda(\gamma)}\big(\wt{X}\big)}\bigg)^{1/2}\cdot\exp\Big(-m\cdot\Xi\big(\gamma,\lambda(\gamma)\big)\Big)\,,
\end{equation}
where $X\sim\textnormal{Binom}(k,\gamma)$ and $\wt{X}\sim \nu_{\gamma,\lambda(\gamma)}$. Lemma \ref{lem:lagrange} further shows that $\big|\lambda(\gamma)\big|\lesssim_k n^{-1/3}$, thus we have
\begin{equation}\label{eq:var:estimate}
\Var_{\gamma,\lambda(\gamma)}\big(\wt{X}\big)\asymp_k \Var_{\gamma}\big(X\,\big|\,1\leq X\leq k-1\big)\asymp_k \Var_{\gamma}(X)\,,
\end{equation}
where the final estimate holds because $|\gamma-1/2|\leq n^{-1/3}$. Combining with \eqref{eq:gamma:negligible}, it follows that
\begin{equation}\label{eq:firstmo:estimate}
\E Z_{\textsc{col}}=\big(1+o_n(1)\big)\sum_{|\gamma-1/2|\leq n^{-1/3}}\binom{n}{n\gamma}\bp_{\gamma}\asymp_{k,d} n^{-1/2}\sum_{|\gamma-1/2|\leq n^{-1/3}}\exp\big(nG_{\alpha}(\gamma)\big)\,,
\end{equation}
where
\[
G_{\alpha}(\gamma):=H(\gamma)-\alpha \cdot \Xi\big(\gamma,\lambda(\gamma)\big)\,.
\]
Note that by comparing \eqref{eq:bp:estimate} and \eqref{eq:var:estimate} with \eqref{eq:local:clt}, we have $G_{\alpha}(\gamma)\leq F_{\alpha}(\gamma)$ for $|\gamma-1/2|\leq n^{-1/3}$. Also, note that for $\gamma=1/2$, $G_{\alpha}(1/2)=F_{\alpha}(1/2)$ holds since
\[
G_{\alpha}(1/2)=H(1/2)-\alpha\cdot \Xi(1/2,0)=H(1/2)+\alpha\log\big(1-\gamma^k-(1-\gamma)^k\big)\,,
\]
where we used $\lambda(1/2)=0$ by Lemma \ref{lem:lagrange}. Recalling that $\gamma\to F_{\alpha}(\gamma)$ is uniquely maximized at $\gamma=1/2$ with strictly negative second derivative at the maximizer, it follows that the same holds for $\gamma\to G_{\alpha}(\gamma)$. Therefore, combining with \eqref{eq:firstmo:estimate}, we have
\[
\E Z_{\textsc{col}}\asymp_{k,d} \exp\big(nG_{\alpha}(1/2)\big)=\E Z_{\textsc{nae}}\,,
\]
which concludes the proof.
\end{proof}
\section{Proof of Proposition \ref{prop:bp} for $k \ge 4$}
\label{sec:3}
In this section, we prove Proposition \ref{prop:bp} for $k\geq 4$, which can be split into the following two lemmas. In Section \ref{sec:3.1}, we prove Lemma \ref{lem:partone}  which guarantees the existence and the uniqueness of the \textsc{bp} fixed point for $k \ge 4$.

\begin{lemma}  \label{lem:partone}
     For $k \ge 4$ and $d\in [d_{\textnormal{lbd}}(k), d_{\textnormal{ubd}}(k)]$, there exists a unique solution to $\Psi_d(x)=x$ in the range $x\in [\frac{1}{2} - \frac{1}{2^k}, \frac{1}{2}]$.
\end{lemma}
By Lemma \ref{lem:partone}, the function $d\to \boldsymbol{\prescript{\star}{}{\Phi}}(d)$ is well-defined. In Section \ref{sec:3.2}, we prove Lemma \ref{lem:parttwo} which guarantees that $d_\star(k)$ is well-defined for $k \ge 4$. 
\begin{lemma}\label{lem:parttwo}
For $k\geq 4$, the function $d\to \boldsymbol{\prescript{\star}{}{\Phi}}(d)$ is continuous for $d\in [d_{\textnormal{lbd}}(k), d_{\textnormal{ubd}}(k)]$. Further, $\boldsymbol{\prescript{\star}{}\Phi}(d_{\textnormal{lbd}}(k))>0$ and $\boldsymbol{\prescript{\star}{}\Phi}(d_{\textnormal{ubd}}(k))<0$ hold.
\end{lemma}
\begin{proof}[Proof of Proposition \ref{prop:bp} for $k\ge 4$]
This is immediate from Lemma \ref{lem:partone} and Lemma \ref{lem:parttwo}.
\end{proof}

\subsection{Proof of Lemma \ref{lem:partone}}
\label{sec:3.1}
Recall the variable \textsc{bp} recursion $\dot{\Psi}$ and the clause \textsc{bp} recursion $\hat{\Psi}$ defined in \eqref{eq:def:var:clause:BP}. To prove the uniqueness of the \textsc{bp} fixed point, we show that the \textsc{bp} recursion $\Psi_d\equiv \dot{\Psi}\circ \hat{\Psi}$ is a contraction for $k\geq 4$.

\begin{lemma}
\label{lem:derivativebound}
For $k\ge 4$ and $d\in [d_{\textnormal{lbd}}(k), d_{\textnormal{ubd}}(k)]$, $\big|(\Psi_d)^\prime (x)\big|<1$ holds uniformly over $x\in [\frac{1}{2}-\frac{1}{2^k},\frac{1}{2}]$.
\end{lemma}

\begin{proof} 
Throughout, we let $x\in[1/2-1/2^k,1/2]$ and denote $v=\hat{\Psi}(x)$. We first consider $k\geq 5$. Observe that the derivative of the clause \textsc{bp} recursion can simply be bounded in absolute value by
     \begin{equation}\label{eq:bound:clause:bp}
     \big|(\hat{\Psi})^\prime(x)\big|=\frac{(k-1)x^{k-2}}{(1-x^{k-1})^2} \leq \frac{(k-1)\cdot 2^{-k+2}}{(1-2^{-k+1})^2}=\frac{4(k-1)}{2^k (1-2^{-k+1})^2}\,, 
     \end{equation}
     where the inequality holds since $x\to \frac{x^{k-2}}{(1-x^{k-1})^2}$ is increasing. Similarly, we bound the derivative of the variable \textsc{bp} recursion:
    % \begin{equation}\label{eq:clause:BP:derivative:tech:1}
    % \big|(\dot{\Psi})^\prime(v)\big|=\frac{(d-1)v^{d-2}}{(2-v^{d-1})^2} \leq \frac{(d-1)v_0^{d-2}}{(2-v_0^{d-1})^2}<\frac{(d-1)v_0^{d-1}}{(2-v_0^{d-1})^2}\,,
    % \end{equation}
    \begin{equation}\label{eq:clause:BP:derivative:tech:1}
    \big|(\dot{\Psi})^\prime(v)\big|=\frac{(d-1)v^{d-2}}{(2-v^{d-1})^2} \leq \frac{(d-1)v_0^{d-2}}{(2-v_0^{d-1})^2} \leq \frac{(d-1)v_0^{d-2}}{(2-v_0^{d-2})^2}\,,
    \end{equation}
    where we denoted $v_0:=\hat{\Psi}(x_0)$ for $x_0=1/2-1/2^k$. The first inequality holds because $x \to \hat{\Psi}(\cdot)$ is decreasing on $[1/2-1/2^k,1/2]$, and the last inequality holds since $v_0<1$. To this end, we upper bound $v_0^{d-2}$ by 
    \begin{equation}\label{eq:clause:BP:derivative:tech:2}
     v_0^{d-2}=\left(1-\frac{x_0^{k-1}}{1-x_0^{k-1}}\right)^{d-2}\leq (1-x_0^{k-1})^{d-2}\leq e^{-(d-2)x_0^{k-1}}\,.
    \end{equation}
    Note that $x_0^{k-1}=\left(\frac{1}{2}\right)^{k-1} \left(1-\frac{2}{2^k}\right)^{k-1} \geq \left(\frac{1}{2}\right)^{k-1}\left(1-\frac{2(k-1)}{2^k}\right)$ and $d\geq (2^{k-1}-2)k \log 2$ hold, thus we can lower bound $(d-2)x_0^{k-1}$ by
    \[
    (d-2)x_0^{k-1}\geq \left(k\log 2 -\frac{4k\log 2+4}{2^k}\right)\cdot \left(1-\frac{2(k-1)}{2^k}\right)\,.
    \]
    Thus, combining with \eqref{eq:clause:BP:derivative:tech:2} shows that
    \begin{equation}\label{eq:bound:v0:power}
        v_0^{d-2}\leq 2^{-k}e^{\eps_k}\,,\quad\textnormal{where}\quad \eps_k:= \frac{2(k-1)k\log 2}{2^k}+\frac{4k\log2+4}{2^k}\left(1-\frac{2(k-1)}{2^k}\right)\,.
    \end{equation}
    Plugging this bound into \eqref{eq:clause:BP:derivative:tech:1}, we have

    \[
    |(\dot{\Psi})^\prime(v)|< (d-1) \frac{v_0^{d-2}}{(2-v_0^{d-2})^2}\leq (2^{k-1} k \log 2-1)\cdot \frac{2^{-k}\cdot e^{\eps_k}}{(2-2^{-k}e^{\eps_k})^2} \,.
    \]
   Combining with the contraction of clause \textsc{bp} recursion in \eqref{eq:bound:clause:bp}, we have
    \[|(\Psi_{d})^\prime(x)|\leq \alpha_k:= \frac{2k(k-1)\log 2}{2^k}\cdot\left(1-\frac{1}{2^{k-1} k \log 2}\right) \cdot\frac{e^{\eps_k}}{(1-2^{-k+1})^2(2-2^{-k}e^{\eps_k})^2}\,.\]
    By comparing $\eps_k$ and $\eps_{k+1}$ for $k\geq 5$, it can be easily checked that $k\to \eps_k$ is decreasing, and the same holds for $k\to \frac{2k(k-1)\log 2}{2^k}\cdot\left(1-\frac{1}{2^{k-1} k \log 2}\right)$. Thus, $k\to\alpha_k$ is decreasing for $k\geq 5$. Furthermore, $\alpha_5$ can be calculated up to arbitrary precision (e.g. by Mathematica), which satisfies $\alpha_5<0.99<1$. Consequently, $|(\Psi_{d})^\prime(x)|<1$ holds for $k\geq 5$.
    
    The case where $k=4$ is more delicate, and the previous strategy of bounding the derivative of clause and variable \textsc{bp} recursions separately no longer is successful. To this end, we bound $(\Psi_d)^\prime(x)$ directly. If we denote $v=\hat{\Psi}_k(x)$, then
   \begin{align*}
   \big|(\Psi_d)^\prime(x)\big|=\big|(\hat{\Psi})^\prime(x)|\cdot |(\dot{\Psi})^\prime(v)\big| &= \frac{(k-1)(d-1)v^{d-2}}{(2-v^{d-1})^2} \cdot \frac{x^{k-1}}{(1-x^{k-1})^2} \cdot \frac{1}{x}\,.
   \end{align*}
    Since $v\equiv \hat{\Psi}_k(x)\equiv \frac{1-2x^{k-1}}{1-x^{k-1}}$, rearranging gives $x^{k-1}=\frac{1-v}{2-v}$. Substituting this in for $x^{k-1}$, we have that
   \begin{equation}\label{eq:compute:k:4:derivative}
    \big|(\Psi_d)^\prime(x)\big|=(k-1)(d-1)\cdot\frac{v^{d-2}(2-v)(1-v)}{(2-v^{d-1})^2}\cdot \frac{1}{x}\,.
    \end{equation}
   We now claim that $v\to \frac{v^{d-2}(2-v)(1-v)}{(2-v^{d-1})^2}$ is increasing for $v\in [\hat{\Psi}_{4}(1/2),\hat{\Psi}_4(1/2-1/2^4)]$ and $d\in [24\log 2, 32\log 2]$ (recall that $24\log 2>16.7\equiv d_{\textnormal{lbd}}(4)$ holds). Since $v\to (2-v^{d-1})^2$ is decreasing, it suffices to show that $v\to v^{d-2}(2-v)(1-v)$ is increasing. Note that
   \[
   \frac{\de}{\de v}\Big(v^{d-2}(2-v)(1-v)\Big)=(dv^2-3(d-1)v+2(d-2))v^{d-3}>0\iff d>\frac{4-3v}{(2-v)(1-v)}\,.
   \]
   Note that $v\to \frac{4-3v}{(2-v)(1-v)}$ is increasing since its derivative is given by $\frac{3v^2-8v+6}{(2-v)^2(1-v)^2}>0$. Thus, to prove our claim, it suffices to check that for $d_0:=24\log 2$ and $v_0= \hat{\Psi}_4(1/2-1/2^4)$ that $d_0>\frac{4-3v_0}{(2-v_0)(1-v_0)}$ holds. By a direct calculation, $v_0=3410/3753<0.91$ and $24\log 2>16>\frac{4-3\cdot 0.91}{(2-0.91)(1-0.91)}$ holds, thus the claim that $v\to \frac{v^{d-2}(2-v)(1-v)}{(2-v^{d-1})^2}$ is increasing is proven for $d,v$ in the regime of interest.

   Note that $x \to v=\hat{\Psi}_4(x)$ is decreasing, thus \eqref{eq:compute:k:4:derivative} and our previous claim shows that for all $x_0\leq x\leq 1/2$, where $x_0=1/2-1/2^4$, we have
   \[
    \big|(\Psi_d)^\prime(x)\big|\leq (d-1)(k-1)\frac{v_0^{d-2}(2-v_0)(1-v_0)}{(2-v_0^{d-1})^2}\cdot \frac{1}{x_0}\,,
    \]
   where $v_0=\hat{\Psi}_4(x_0)=3410/3753$. We next show that the right hand side as a function of $d\in [24\log 2, 32\log 2]$ is decreasing: since $d\to (2-v_0^{d-1})^2$ is increasing, it suffices to show that $d\to (d-1)v_0^{d-2}$ is decreasing. Note that
    \[
   \frac{\de }{\de d}\Big((d-1)v_0^{d-2}\Big)=v_0^{d-2}\Big(1-(d-1)\log \big(1/v_0\big)\Big)<0 \iff d>\frac{1}{\log(1/v_0)}+1\,,
    \]
    and it can be verified that $24\log 2 >16>1/\log(3753/3410)+1$ holds. Therefore, for $k=4$, it follows that for $d_0=24\log 2$,
    \[
     \big|(\Psi_d)^\prime(x)\big|\leq 3(d_0 -1)\frac{v_0^{d_0-2}(2-v_0)(1-v_0)}{(2-v_0^{d_0-1})^2}\cdot \frac{1}{x_0}\,.
     \]
     The right hand side can be computed to arbitrary precision (e.g. by Mathematica), it can be verified that $3(d_0 -1)\frac{v_0^{d_0-2}(2-v_0)(1-v_0)}{(2-v_0^{d_0-1})^2}\cdot \frac{1}{x_0}<0.9<1$. This concludes the proof for the case $k=4$.
\end{proof}
In the proof of Lemma \ref{lem:derivativebound}, we did not use the adjustment for $d_{\textnormal{lbd}}(4)\equiv 16.7>24\log 2$. That is,  $\max_{\frac{1}{2}-\frac{1}{2^4}\leq x \leq \frac{1}{2}}\big|(\Psi_d)^\prime (x)\big|<1$ holds for $d\in [24\log 2, 32\log 2]$. The adjustment $d_{\textnormal{lbd}}(4)\equiv 16.7$ is needed for the following lemma, which guarantees the existence of the solution to $\Psi_d(x)=x$.
% Note that since $\dot{\Psi}_k(v)<1/2$ for every $v>0$, we have that $\Psi_d(1/2)<1/2$. Thus, it suffices to show $\Psi_{d}(\frac{1}{2}-\frac{1}{2^k}) >\frac{1}{2}-\frac{1}{2^k}$.
\begin{lemma} 
\label{lem:leftendpoint}

    $\Psi_{d}(\frac{1}{2}-\frac{1}{2^k}) >\frac{1}{2}-\frac{1}{2^k}$ holds for $k\geq 4$ for $d\in [d_{\textnormal{lbd}}(k),d_{\textnormal{ubd}}(k)]$.
\end{lemma}
\begin{proof} 
    Let $v_0\equiv v_0(k)=\hat{\Psi}\left(\frac{1}{2}-\frac{1}{2^k}\right)$ as before. Then, from the definition of $\dot{\Psi}, \hat{\Psi}$ in \eqref{eq:def:var:clause:BP}, $\Psi_{d}(\frac{1}{2}-\frac{1}{2^k}) >\frac{1}{2}-\frac{1}{2^k}$ is equivalent to  $v_0^{d-1}<\frac{4}{2^k+2}$, which we aim to show for $k\geq 4$. We start with the case $k\geq 5$. We have shown in \eqref{eq:bound:v0:power} that $v_0^{d-2}\leq 2^{-k}e^{\eps_k}$, holds, and by an analogous proof, $v_0^{d-1} \leq 2^{-k}e^{\beta_k}$ holds, where $\beta_k\equiv \eps_k-\frac{1}{2^{k-1}}\big(1-\frac{2(k-1)}{2^k}\big)$. Thus, it suffices to show that
    \begin{equation*}
    e^{\beta_k}\left(1+\frac{1}{2^{k-1}}\right)<4\,,\quad\textnormal{where}\quad \beta_k\equiv\frac{2(k-1)k\log 2}{2^k}+\frac{4k\log2+2}{2^k}\left(1-\frac{2(k-1)}{2^k}\right)\,.
    \end{equation*}
    For $k=5$, $e^{\beta_5}(1+1/2^4)$ can be computed to arbitrary precision (e.g. by Mathematica), and it can be numerically verified that $e^{\beta_5}(1+1/2^4)<3.7$. Further, $k\to \beta_k$ is decreasing by comparing $\beta_k$ and $\beta_{k+1}$, thus this concludes the proof for $k\geq 5$.

    Next, we consider the case $k=4$. Since $d\to v_0^{d-1}$ is maximized at $d=d_{\textnormal{lbd}}(4) \equiv 16.7$, it suffices to show that $v_0^{15.7}\leq \frac{2}{9}$ holds, where $v_0\equiv \hat{\Psi}_4(1/2-1/2^4)=3410/3753$. Since $v_0^{15.7}=(3410/3753)^{15.7}$ can be computed to arbitrary precision (e.g. by Mathematica), it can be checked that $v_0^{15.7}=(3410/3753)^{15.7} <0.2221 < \frac{2}{9}$ holds, so this concludes the proof.
\end{proof}

\begin{proof}[Proof of Lemma \ref{lem:partone}]
By Lemma \ref{lem:leftendpoint}, $\Psi_{d}(\frac{1}{2}-\frac{1}{2^k}) >\frac{1}{2}-\frac{1}{2^k}$ holds for $k\geq 4$. Note that $\Psi_d(1/2)<1/2$ holds since $\dot{\Psi}(x)<1/2$ holds for any $x\geq 0$. Thus, since $x\to \Psi_d(x)$ is continuous and differentiable, intermediate value theorem guarantees the existence of the solution to $\Psi_d(x)=x$ for $x\in [\frac{1}{2}-\frac{1}{2^k},\frac{1}{2}]$. Moreover, $\big|(\Psi_d)^\prime(x)\big|<1$ holds uniformly over $x\in [\frac{1}{2}-\frac{1}{2^k},\frac{1}{2}]$ by Lemma \ref{lem:derivativebound}, thus mean value theorem guarantees the uniqueness of the solution to $\Psi_d(x)=x$ for $x\in [\frac{1}{2}-\frac{1}{2^k},\frac{1}{2}]$.
\end{proof}

\subsection{Proof of Lemma \ref{lem:parttwo}}
\label{sec:3.2}
Recall that $\bPhi(d)$ is defined in \eqref{eq:def:Phi} as $\bPhi(d)\equiv \Phi\big(d,x(k,d)\big)$, where $x(k,d)\in [\frac{1}{2}-\frac{1}{2^k},\frac{1}{2}]$ is the solution to $\Psi_d(x)=x$, and we defined the function $\Phi(d,x)$ by
 \begin{equation}\label{eq:def:Phi:d:x}
 \Phi(d,x)\equiv \Phi_k(d,x):=-\log(1-x)-d(1-k^{-1}-d^{-1})\log(1-2x^k)+(d-1)\log(1-x^{k-1})\,.
 \end{equation}
 To prove $\bPhi\big(d_{\textnormal{lbd}}(k)\big)>0$ and $\bPhi\big(d_{\textnormal{ubd}}(k)\big)<0$, we show respectively in Lemmas \ref{lem:leftendpoint:Phi} and \ref{lem:rightendpoint:Phi} that $\Phi\big(d_{\textnormal{lbd}}(k), x\big)>0$ and  $\Phi\big(d_{\textnormal{ubd}}(k), x\big)<0$ hold uniformly over $x\in [\frac{1}{2}-\frac{1}{2^k},\frac{1}{2}]$.
\begin{lemma} \label{lem:leftendpoint:Phi}
For $k\geq 4$, $\Phi(d_{\textnormal{lbd}}(k), x)>0$  holds uniformly over $x\in [\frac{1}{2}-\frac{1}{2^k},\frac{1}{2}]$.
\end{lemma}

\begin{proof}
Note that rearranging $\Phi(d,x)$ gives
\begin{equation}\label{eq:lem:leftendpoint:Phi:tech}
\begin{split}
\Phi(d,x)
&=-\log(1-x)-d\big((1-k^{-1})\log(1-2x^k)-\log(1-x^{k-1})\big)+\log(1-2x^k)-\log(1-x^{k-1})\\
&\geq -\log(1-x)-d\big((1-k^{-1})\log(1-2x^k)-\log(1-x^{k-1})\big)\,,
\end{split}
\end{equation}
where the inequality holds since $\log(1-2x^k)\geq \log (1-x^{k-1})$ holds for $x\in [0,1/2]$. Note that the first term in the right hand side $x\to -\log (1-x)$ is convex, so the linear approximation at $x=1/2$ shows that $-\log (1-x)\geq \log 2+ 2(x-1/2)$ holds. Further, the function $x \to (1-k^{-1})\log(1-2x^k)-\log(1-x^{k-1})$ is increasing since 
\[
\frac{\de}{\de x}\Big((1-k^{-1})\log(1-2x^k)-\log(1-x^{k-1})\Big)=\frac{(k-1)x^{k-2}(1-2x)}{(1-2x^k)(1-x^{k-1})}\geq 0\,.
\]
Thus, the right hand side in \eqref{eq:lem:leftendpoint:Phi:tech} for $d=d_{\textnormal{lbd}}(k)$ can further lower bounded by
\begin{equation}\label{eq:left:Phi:tech:2}
\begin{split}
\Phi\big(d_{\textnormal{lbd}}(k),x\big)
&\geq \log 2+ 2\big(x-1/2\big)+\frac{d_{\textnormal{lbd}}(k)}{k}\cdot \log\big(1-2^{-k+1}\big)\\
&\geq \log 2-2^{-k+1}+\frac{d_{\textnormal{lbd}}(k)}{k}\cdot \log\big(1-2^{-k+1}\big)=:F(k)\,,
\end{split}
\end{equation}
where we used $x\geq 1/2-1/2^k$ in the last inequality. Using the inequality $\log(1-a)\geq -a-\frac{a^2}{2}-\frac{a^3}{2}$ for $a=2^{-k+1}\leq \frac{1}{8}$, we have that for $k\geq 5$ that

\[
F(k)=\log 2 - 2^{-k+1}+(2^{k-1}-2)\log 2\cdot \log (1-2^{-k+1})\geq\frac{1}{2^k}\left(3\log 2 -2 -\frac{6\log2}{2^k}+\frac{8\log 2}{2^{2k}}\right).
\]

For $k\geq 6$, the right hand side above is positive since 
$3\log 2-2-\frac{3\log 2}{32}>0.01$, thus \eqref{eq:left:Phi:tech:2} shows that $\Phi\big(d_{\textnormal{lbd}}(k),x\big)>0$ holds for $k\geq 6$ and $x\in [\frac{1}{2}-\frac{1}{2^k},\frac{1}{2}]$. For $k\in \{4,5\}$, we can explicitly calculate $F(k)$ by $F(4)\equiv \log 2 -1/8+(16.7/4)\log\left(7/8\right)>0.01>0$, and $F(5)\equiv \log 2 -1/16+14\log2 \cdot \log\left(15/16\right)>0.004>0$, thus \eqref{eq:left:Phi:tech:2} again concludes the proof for $k\in \{4,5\}$.

\end{proof}
\begin{lemma}\label{lem:rightendpoint:Phi}
For $k\geq 4$, $\Phi\big(d_{\textnormal{ubd}}(k),x\big)<0$ holds uniformly over $x\in [\frac{1}{2}-\frac{1}{2^k},\frac{1}{2}]$. 
\end{lemma}
\begin{proof} 
We first claim that for $k\geq 5$, the function $x\to \Phi\big(d_{\textnormal{ubd}}(k),x\big)$ is increasing for $x\in [\frac{1}{2}-\frac{1}{2^k},\frac{1}{2}]$ and $d_{\textnormal{ubd}}(k)\equiv 2^{k-1}k\log 2$. A direct calculation shows that
 \begin{equation}\label{eq:Phi:derivative:x}
 \begin{split}
 \frac{ \partial \Phi}{\partial x}\big(d_{\textnormal{ubd}}(k),x\big)
 &=\frac{1}{1-x}-(2^{k-1}k\log 2-1)(k-1)\cdot \frac{x^{k-2}(1-2x)}{(1-x^{k-1})(1-2x^k)}-\frac{2x^{k-1}}{1-2x^k}\\
 &\geq \frac{1}{\frac{1}{2}+\frac{1}{2^k}}-(2^{k-1}k\log 2-1)(k-1)\cdot \frac{x^{k-2}(1-2x)}{(1-x^{k-1})(1-2x^k)}-\frac{4}{2^k-2}\,,
 \end{split}
    \end{equation}
    where the inequality holds since $x\to (1-x)^{-1}$ increasing, so it is minimized at $x=1/2+1/2^k$, and $x\to 2x^{k-1}/(1-2x^k)$ is increasing, so it is maximized at $x=1/2$. Further, it is straightforward to check that $x\to x^{k-2}(1-2x)$ is decreasing for $x\in [\frac{1}{2}-\frac{1}{2^k},\frac{1}{2}]$, thus it is maximized at $x=1/2-1/2^k$. Also, $x\to (1-x^{k-1})(1-2x^k)$ is minimized at $x=1/2$. Thus, by plugging in these bounds, we can further bound
\begin{equation}
\begin{split}
\frac{ \partial \Phi}{\partial x}\big(d_{\textnormal{ubd}}(k),x\big)
&\geq 2-\left(\frac{2}{2^{k-1}+1}+\frac{4}{2^k-2}+\frac{(2^{k-1}k\log 2-1)(k-1)}{2^{2k-3}}\cdot\left(1-\frac{1}{2^{k-1}}\right)^{k-4}\right)\\
&\geq 2-\left(\frac{2}{2^{k-1}+1}+\frac{4}{2^k-2}+\frac{(2^{k-1}k\log 2-1)(k-1)}{2^{2k-3}}\right)=:2-G(k)\,.
\end{split}
\end{equation}
Note that the function $k\to G(k)$ is increasing for $k\geq 5$. Furthermore, using the bound $\log 2<0.7$, we can bound $G(5)=\frac{2}{17}+\frac{2}{15}+\frac{80\log 2-1}{32}<1.97<2$. Therefore, $\frac{ \partial \Phi}{\partial x}\big(d_{\textnormal{ubd}}(k),x\big)>0$ holds for $k\geq 5$ and $x\in [\frac{1}{2}-\frac{1}{2^k},\frac{1}{2}]$, which proves our first claim.

Consequently, for the case $k\geq 5$, it suffices to show that $\Phi(2^{k-1}k\log2,x)<0$ holds for $x=1/2$. A direct calculation gives
    \begin{align}\label{eq:Phi:right:end}
        \Phi\Big(2^{k-1}k\log2,\frac{1}{2}\Big) =\log2+2^{k-1}\log 2 \cdot\log \Big(1-\frac{1}{2^{k-1}}\Big)<0\,,
    \end{align}
    where the inequality holds since $\log(1-a)<-a$ holds for $a\in (0,1)$. This concludes the proof for $k\geq 5$.

    It remains to consider the case $k=4$. For $k=4$, we claim that $x\to \Phi_4\big(d_{\textnormal{ubd}}(4),x)$ is convex in the interval $x\in [\frac{7}{16},\frac{1}{2}]$. From the computation of $\frac{ \partial \Phi}{\partial x}\big(d_{\textnormal{ubd}}(k),x\big)$ in \eqref{eq:Phi:derivative:x}, we can calculate the second derivative by
    \begin{equation}\label{eq:second:derivative}
    \frac {\partial^2 \Phi}{\partial x^2}\big(d_{\textnormal{ubd}}(4),x)=\frac{\de}{\de x}\left(\frac{1}{1-x}-\frac{2x^3}{1-2x^4}\right)+3(32\log 2-1)\cdot\frac{\de}{\de x}\left(\frac{x^2(2x-1)}{(1-x^3)(1-2x^4)}\right)\,.
    \end{equation}
    The first term in the right hand side can be bounded by
    \begin{equation}\label{eq:second:derivative:tech}
    \frac{\de}{\de x}\left(\frac{1}{1-x}-\frac{2x^3}{1-2x^4}\right)=\frac{1}{(1-x)^2}-\frac{6x^2+4x^6}{(1-2x^4)^2}>\frac{1}{(1-\frac{7}{16})^2}-\frac{6\left(\frac{1}{2}\right)^2+4\left(\frac{1}{2}\right)^6}{(1-2\left(\frac{1}{2}\right)^4)^2}>0\,,
    \end{equation}
    where the final inequality is equivalent to $\frac{256}{81}-\frac{100}{49}>0$. The second term can be calculated as
    \[
    \frac{\de}{\de x}\left(\frac{x^2(2x-1)}{(1-x^3)(1-2x^4)}\right)=\frac{x(-16 x^8 +10 x^7 + 4 x^5 - 4 x^4 - x^3 + 6 x - 2)}{(1-x^3)^2(1-2x^4)^2}\,.
    \]
    Note that by neglecting the terms $10x^7+4x^5$ above, we can lower bound
    \[
    -16 x^8 +10 x^7 + 4 x^5 - 4 x^4 - x^3 + 6 x - 2>6\cdot\frac{7}{16}-2-\left(\frac{1}{2}\right)^3-4\left(\frac{1}{2}\right)^4-16\left(\frac{1}{2}\right)^8>0\,,
    \]
    thus $\frac{\de}{\de x}\left(\frac{x^2(2x-1)}{(1-x^3)(1-2x^4)}\right)>0$ holds for $x\in [\frac{7}{16},\frac{1}{2}]$ as well. Therefore, combining with \eqref{eq:second:derivative} and \eqref{eq:second:derivative:tech} finishes the proof of our claim that $x\to \Phi_4\big(d_{\textnormal{ubd}}(4),x)$ is convex in the interval $x\in [\frac{7}{16},\frac{1}{2}]$.

   Thus, by convexity, $x\to \Phi_4\big(d_{\textnormal{ubd}}(4),x)$ is maximized at the end points $x\in \{7/16,1/2\}$, and it suffices to show that $\Phi_4\big(d_{\textnormal{ubd}}(4),7/16)<0$ and $\Phi_4\big(d_{\textnormal{ubd}}(4),1/2)<0$. For $x=7/16$, $\Phi_4\big(d_{\textnormal{ubd}}(4),7/16)$ can be computed to arbitrary precision (e.g. by Mathematica), and it can be checked that $\Phi_4\big(d_{\textnormal{ubd}}(4),7/16)<-0.08<0$. For $x=1/2$, \eqref{eq:Phi:right:end} shows that $\Phi_4\big(d_{\textnormal{ubd}}(4),1/2)<0$ holds. This concludes the proof for the case $k=4$.
\end{proof}

    \begin{proof}[Proof of Lemma \ref{lem:parttwo}]
     By definition, $\bPhi(d)=\Phi\big(d,x(k,d)\big)$ holds, and $(d,x)\to \Phi(d,x)$ is clearly continuous. Thus, in order to show the continuity of $\bPhi(\cdot)$, it suffices to show that $d\to x(k,d)$ is continuous for any fixed $k\geq 4$. To that end, note that the function $\psi(d,x):=\Psi_d(x)-x$ satisfies $\frac{\partial \psi}{\partial x}<0$ by Lemma \ref{lem:derivativebound}. Since $x(k,d)$ is defined to be the root of $\psi(d,\cdot)$, this implies that $d\to x(k,d)$ is continuous by the implicit function theorem. As a consequence, we conclude that $d\to \bPhi(d)$ is continuous. Since  $\boldsymbol{\prescript{\star}{}\Phi}(d_{\textnormal{lbd}}(k))>0$ holds by Lemma \ref{lem:leftendpoint:Phi} and $\boldsymbol{\prescript{\star}{}\Phi}(d_{\textnormal{ubd}}(k))<0$ holds by Lemma \ref{lem:rightendpoint:Phi}, we conclude the proof.
    \end{proof}

    \section{Proof of Proposition \ref{prop:bp} for $k=3$}
    \label{sec:4}
In this section, we prove Proposition \ref{prop:bp} for $k=3$. Previous arguments for $k\geq 4$ in Section~\ref{sec:3} do not work because $\boldsymbol{\prescript{\star}{}{\Phi}}$ is in fact, not well defined for all $d$ in the interval $[6\log 2, 12\log2]$. To resolve this, we instead restrict our attention to $d\in [d_{\textnormal{lbd}}(3), d_{\textnormal{ubd}}(3)]\equiv[6.74, 7.5] \subset [6\log 2, 12\log 2]$. In Section \ref{sec:4.1} we show the following lemma which guarantees the existence and the uniqueness of the \textsc{bp} fixed point for $k =3$. 
 
 \begin{lemma} \label{lem:kthreepartone} For $k=3$, $d\in [d_{\textnormal{lbd}}(3), d_{\textnormal{ubd}}(3)]\equiv[6.74, 7.5]$, there exists a unique solution to $\Psi_d(x)=x$ in the range $x\in [\frac{1}{2} - \frac{1}{2^3}, \frac{1}{2}]$.
 \end{lemma}

By Lemma \ref{lem:kthreepartone}, the function $d\to \boldsymbol{\prescript{\star}{}{\Phi}}(d)$ is well-defined. In Section \ref{sec:4.2}, we prove Lemma \ref{lem:kthreeparttwo} which guarantees that $d_\star(3)$ is well-defined.

\begin{lemma}\label{lem:kthreeparttwo}
For $k = 3$, the function $d\to \boldsymbol{\prescript{\star}{}{\Phi}}(d)$ is continuous for $d\in [d_{\textnormal{lbd}}(3), d_{\textnormal{ubd}}(3)]\equiv[6.74, 7.5]$. Further, $\boldsymbol{\prescript{\star}{}\Phi}(d_{\textnormal{lbd}}(3))>0$ and $\boldsymbol{\prescript{\star}{}\Phi}(d_{\textnormal{ubd}}(3))<0$ hold.
\end{lemma}

\begin{proof}[Proof of Proposition \ref{prop:bp} for $k = 3$]
This is immediate from Lemma \ref{lem:kthreepartone} and Lemma \ref{lem:kthreeparttwo}.
\end{proof}

\subsection{Proof of Lemma \ref{lem:kthreepartone}}
\label{sec:4.1}
Recall the variable \textsc{bp} recursion $\dot{\Psi}$ and the clause \textsc{bp} recursion $\hat{\Psi}$ defined in \eqref{eq:def:var:clause:BP}. To prove the uniqueness of the \textsc{bp} fixed point, we show that the \textsc{bp} recursion $\Psi_d\equiv \dot{\Psi}\circ \hat{\Psi}$ is a contraction for $k=3.$

\begin{lemma} \label{lem:threederivative}
For $d\in [d_{\textnormal{lbd}}(3), d_{\textnormal{ubd}}(3)]\equiv[6.74, 7.5]$, $\big|(\Psi_d)^\prime (x)\big|<1$ holds uniformly over $x\in [\frac{1}{2}-\frac{1}{2^3},\frac{1}{2}]$.
\end{lemma}

\begin{proof}
For $k=3$, a direct calculation gives
\[
(\Psi_{d})^\prime(x)=  \frac{2(d-1) \left(\frac{1-2x^2}{1-x^2} \right)^{d-2}}{\left(2- \left( \frac{1-2x^2}{1-x^2}\right)^{d-1} \right)^2} \cdot \frac{x}{(1-x^2)^2}\,.
\]
Using the inequality $\frac{1-2x^2}{1-x^2} < 1-x^2$ in the denominator above, to prove our goal $\big|(\Psi_{d})^\prime(x)\big|<1$, it suffices to prove that for $d\in [6.74, 7.5]$ and $x\in [7/8,1/2]$,
\[
2(d-1) \left(\frac{1-2x^2}{1-x^2} \right)^{d-2}x < \left(2- (1-x^2)^{d-1} \right)^2 \cdot (1-x^2)^2\,,
\] 
which rearranges to 
\begin{equation}\label{eq:k:3:derivative:goal}
L(d,x):=\frac{\left(1-x^{2}\right)^{d}\left(\left(1-x^{2}\right)^{d-1}-2\right)^{2}}{\left(1-2x^{2}\right)^{d-2}}-2(d-1)x >0\,\quad\textnormal{for}\quad d\in [6.74, 7.5]~~\textnormal{and}~~x\in [7/8,1/2]\,.
\end{equation}
For the rest of the proof, we aim to show \eqref{eq:k:3:derivative:goal}. We first claim that $x\to L(d,x)$ is increasing in the regime of interest. The following observation is useful to prove our claim: suppose we are given differentiable functions $f(x),g(x)$, such that $f(x)\geq 0$, $g(x)\leq 1$, and $g(\cdot)$ is decreasing, i.e. $g^\prime(x)\leq 0$. Then,
\begin{equation}\label{eq:useful}
\big(f(x)g(x)\big)^\prime=f^\prime(x)g(x)+f(x)g^\prime(x)\leq f^\prime(x)\,.
\end{equation}
That is, if we multiply $f(\cdot)$ by a non-negative function by a decreasing function that is less than $1$, denoted by $g(\cdot)$, then $(f\cdot g)^\prime\leq f^\prime$. Using this observation for $f(x)=\frac{\left(1-x^{2}\right)^{d}\left(\left(1-x^{2}\right)^{d-1}-2\right)^{2}}{\left(1-2x^{2}\right)^{d-2}}$ and $g(x)=\frac{(1-2x^2)^{d-2}}{((1-x^2)^2)^{d-2}}$, which is less than $1$ and decreasing, we have that
\begin{equation}\label{eq:derivative:L}
\begin{split}
\frac{\partial L}{\partial x}(d,x)
&\geq \frac{\de}{\de x}\Bigg(\frac{\left(\left(1-x^{2}\right)^{d-1}-2\right)^{2}}{\left(1-x^{2}\right)^{d-4}}\Bigg)-2(d-1)\\
&=x\cdot\left(\frac{8(d-4)}{(1-x^2)^{d-3}}+24(1-x^2)^2-2(d+2)(1-x^2)^{d+1}-\frac{2(d-1)}{x}\right)\,.
\end{split}
\end{equation}
Abbreviating $d_{\lbd}\equiv 6.74, d_{\ubd}\equiv 7.5$, and $x_{\lbd}\equiv 3/8, x_{\ubd}\equiv 1/2$, we crudely bound
\[
\begin{split}
&\frac{8(d-4)}{(1-x^2)^{d-3}}+24(1-x^2)^2-2(d+2)(1-x^2)^{d+1}-\frac{2(d-1)}{x}\\
&\geq \frac{8(d_{\lbd}-4)}{(1-x_{\lbd}^2)^{d_{\lbd}-3}}+24(1-x_{\ubd}^2)^2-2(d_{\ubd}+2)(1-x_{\lbd}^2)^{d_{\lbd}+1}-\frac{2(d_{\ubd}-1)}{x_{\lbd}}\,.
\end{split}
\]
The right hand side above is a combination of fractions and powers of the numbers $6.74,7.5,3/8,1/2$, thus can be computed up to arbitrary precision (e.g. by Mathematica), and it can be checked that the right hand side above is greater than $10$. Therefore, combining with \eqref{eq:derivative:L}, this finishes the proof of our claim that $x\to L(d,x)$ is increasing in the regime of interest.

Since $L(d,\cdot)$ is increasing, to prove our goal \eqref{eq:k:3:derivative:goal}, it remains to prove that $L(d,3/8)>0$ for $d\in [6.74, 7.5]$. By a direct calculation, we have
\begin{equation}\label{eq:L}
L\Big(d\,,\,\frac{3}{8}\Big)=\frac{\left(55/64\right)^{3d-2}}{\left(23/32\right)^{d-2}}-\frac{4\left(55/64\right)^{2d-1}}{\left(23/32\right)^{d-2}}+\frac{4(55/64)^{d}}{(23/32)^{d-2}}-\frac{3(d-1)}{4}\,. 
\end{equation}
We next claim that the $d\to L(d, 3/8)$ is convex for any $d>0$. Recalling that $\frac{\de^2 \gamma^d}{\de d^2}=\gamma^d(\log \gamma)^2$ for $\gamma>0$, we can lower bound the second derivative by neglecting the first term in the right hand side:
\[
\frac{\de^2 L}{\de d^2}\Big(d\,,\,\frac{3}{8}\Big)\geq 4\left(\frac{32}{23}\right)^2\left(\frac{(55/64)^{d}}{(23/32)^{d}}\left(\log\left(\frac{55/64}{23/32}\right)\right)^2-\frac{\left(55/64\right)^{2d}}{\left(23/32\right)^{d}}\left(\log\left(\frac{(55/64)^2}{23/32}\right)\right)^2\cdot\frac{64}{55}\right)
\]
It can be numerically verified (e.g. by Mathematica) that $\left(\log\left(\frac{55/64}{23/32}\right)\right)^2>0.01>\left(\log\left(\frac{(55/64)^2}{23/32}\right)\right)^2\cdot\frac{64}{55}$. Further, we have $\frac{(55/64)^{d}}{(23/32)^{d}}>\frac{\left(55/64\right)^{2d}}{\left(23/32\right)^{d}}$. Thus, the inequality above proves our claim that $d\to L(d,3/8)$ is convex for any $d>0$.

Now, since $L(d,3/8)$ in \eqref{eq:L} can be computed to arbitrary precision (e.g. by Mathematica), it can be numerically verified that $L(6.74,3/8)>0.001>0$ while $L(6,3/8)<-0.2<0$ holds. Therefore, $L(d,3/8)>0$ holds for $d>6.74$ since $L(\cdot, 3/8)$ is convex. Since we have shown that $L(d,\cdot)$ is increasing, this concludes the proof our goal \eqref{eq:k:3:derivative:goal}.
\end{proof}

\begin{lemma} For $k=3$ and $d\in [d_{\lbd}(3),d_{\ubd}(3)]\equiv [6.74,7.5]$, $\Psi_d(\frac{1}{2}-\frac{1}{2^3}) >\frac{1}{2}-\frac{1}{2^3}$ holds. 
\label{lem:threeleftendpoint}
\end{lemma}

\begin{proof}
By a direct calculation, we have
$\Psi_d(3/8)= \big(1-(46/55)^{d-1}\big)/ \big(2-(46/55)^{d-1}\big)$, thus $\Psi_d(3/8)>3/8$ is equivalent to $\left(\frac{55}{46}\right)^{d-1}>\frac{5}{2}$. Since $d\to \left(\frac{55}{46}\right)^{d-1}$ is increasing, it suffices to check this for $d=6.74$. It can be checked numerically (e.g. by Mathematica) that $\left(\frac{55}{46}\right)^{5.74}>2.7$, thus $\Psi_d(3/8)>3/8$ holds for any $d\in [6.74,7.5]$.
\end{proof}

\begin{proof}[Proof of Lemma \ref{lem:kthreepartone}]
By Lemma \ref{lem:threeleftendpoint}, $\Psi_{d}(\frac{1}{2}-\frac{1}{2^3}) >\frac{1}{2}-\frac{1}{2^3}$ holds. Note that $\Psi_d(1/2)<1/2$ holds since $\dot{\Psi}(x)<1/2$ holds for any $x\geq 0$. Thus, since $x\to \Psi_d(x)$ is continuous and differentiable, intermediate value theorem guarantees the existence of the solution to $\Psi_d(x)=x$ for $x\in [\frac{1}{2}-\frac{1}{2^3},\frac{1}{2}]$. Moreover, $\big|(\Psi_d)^\prime(x)\big|<1$ holds uniformly over $x\in [\frac{1}{2}-\frac{1}{2^3},\frac{1}{2}]$ by Lemma \ref{lem:threederivative}, thus mean value theorem guarantees the uniqueness of the solution to $\Psi_d(x)=x$ for $x\in [\frac{1}{2}-\frac{1}{2^3},\frac{1}{2}]$.
\end{proof}
    \subsection{Proof of Lemma \ref{lem:kthreeparttwo}}
    \label{sec:4.2}
    For $k\geq 4$, we have proven $\bPhi\big(d_{\lbd}(k)\big)>0$ by showing that $\Phi\big(d_{\lbd}(k),x\big)$, defined in \eqref{eq:def:Phi:d:x}, is uniformly positive for $x\in [\frac{1}{2}-\frac{1}{2^k},\frac{1}{2}]$ in Lemma \ref{lem:leftendpoint}. Unfortunately, the same does not hold for $k=3$. That is, it is not true for $k=3$ that $\Phi(d_{\lbd}(3),x)$ is uniformly positive for $x\in [\frac{1}{2}-\frac{1}{2^3},\frac{1}{2}]$. Instead, we prove $\bPhi(6.74)>0$ by proving a refined estimates for $x\equiv x(3,6.74)$, the solution to \eqref{eq:BP}.

    \begin{lemma} For $k=3$, we have $\boldsymbol{\prescript{\star}{}{\Phi}}(6.74)>0$.
    \label{lem:threephipositive}
    \end{lemma}

    \begin{proof}
    By Lemma \ref{lem:kthreepartone}, there exists a unique solution $x_{\circ}=x(3,6.74)$ to $\Psi_d(x)=x$ for $k=3$ and $d=6.74$. By the uniqueness guaranteed by Lemma \ref{lem:threederivative}, if there exists $a,b\in [3/8, 1/2], a<b$ such that $\Psi_d(a)>a$ and $\Psi_d(b)<b$, then $x_{\circ}\in [a,b]$ holds. By taking $a=0.4464$ and $b=0.45$, it can be numerically verified that $\Psi_{6.74}(0.4464)>0.44645$ and $\Psi_{6.74}(0.45)<0.449$ holds, thus we have $x_{\circ}\in [0.4464,0.45]$.

        We now prove that for $k=3$ and $d=6.74$, the function $x\to \Phi(6.74,x)$ is increasing for $x\in [0.44,0.45]$, where $\Phi$ is defined in \eqref{eq:def:Phi:d:x}. By a direct calculation,
    \begin{align*}
    \Phi(6.74,x) =  -\log(1-x)- \frac{262}{75} \log(1-2x^{3})+\frac{287}{50} \log(1-x^{2})\,.
    \end{align*}
    Differentiating in $x$ gives that for $x\in [0.44,0.45]$
    \[
    \begin{split}
    \frac{\partial\Phi}{\partial x}(6.74,x)
    &= \frac{1}{1-x}+\frac{524}{25}\cdot\frac{x^2}{1-2x^3}-\frac{287}{25}\cdot\frac{x}{1-x^2}\\
    &\geq \frac{1}{1-0.44}+\frac{524}{25}\cdot \frac{(0.44)^2}{1-2(0.44)^3}-\frac{287}{25}\cdot\frac{0.45}{1-(0.45)^2}>0.1\,,
    \end{split}
    \]
    thus $x\to \Phi(6.74,x)$ is increasing for $x\in [0.44,0.45]$.

    As a consequence, it follows that $\bPhi(6.74)\geq \inf_{x\in [0.4464,0.45]}\Phi(6.74,x)=\Phi(6.74,0.4464)$ holds. Further, $\Phi(6.74, 0.4464)$ can be calculated up to arbitrary precision (e.g. by Mathematica), and it can be checked that $\Phi(6.74, 0.4464)>4\cdot 10^{-5}>0$, which concludes the proof. 
    \end{proof}
To show that $\bPhi(7.5)>0$ holds for $k=3$, we use a similar strategy as in the proof of Lemma \ref{lem:threephipositive}.
   
\begin{lemma}\label{lem:threephinegative}
For $k=3$, we have $\bPhi(7.5)<0$ holds.
\end{lemma}

\begin{proof}
    Let $x_{\circ}^\prime\equiv x(3,7.5)$ be the unique solution to $\Psi_d(x)=x$ for $k=3$ and $d=7.5$ (cf. Lemma \ref{lem:kthreepartone}). By taking $a=0.46$ and $b=0.48$, it can be numerically verified that $\Psi_{7.5}(a)<a$ and $\Psi_{7.5}(b)>b$ holds, thus by uniqueness, we have $x_{\circ}^\prime \in [0.46,0.48]$.

    We now prove that for $k=3$ and $d=7.5$, the function $x \to \Phi(7.5,x)$ is increasing for $x\in [0.46, 0.48]$. By definition of $\Phi$ in \eqref{eq:def:Phi:d:x}, we have
\begin{align*}
\boldsymbol{\prescript{\star}{}{\Phi}}(7.5) =  -\log(1-x)- 4 \log(1-2x^{3})+6.5\log(1-x^{2})\,.
\end{align*}
 Differentiating in $x$ gives that for $x\in [0.46,0.48]$
    \[
    \begin{split}
    \frac{\partial\Phi}{\partial x}(6.74,x)
    &= \frac{1}{1-x}+\frac{24x^{2}}{1-2x^{3}}-\frac{13x}{1-x^{2}}= -\frac{2x^{3}-24x^{2}+12x-1}{(x-1)(x+1)(2x^3-1)}\\
    &\geq \frac{1}{1-0.46}+\frac{24\cdot (0.46)^2}{1-2(0.46)^3}-\frac{13\cdot 0.48}{1-(0.48)^2}>0.04>0\,,
    \end{split}
    \]
 thus $x\to \Phi(7.5,x)$ is increasing for $x\in [0.46,0.48]$.

Consequently, it follows that $\bPhi(7.5)\leq \sup_{x\in [0.46,0.48]}\Phi(7.5,x)=\Phi(7.5,0.48)$. Further, $\Phi(7.5,0.48)$ can be calculated up to arbitrary precision (e.g. by Mathematica), and the inequality $\Phi(7.5,0.48)<-0.04<0$ can be checked, which concludes the proof.
\end{proof}
 \begin{proof}[Proof of Lemma \ref{lem:kthreeparttwo}]
     By definition, $\bPhi(d)=\Phi\big(d,x(3,d)\big)$ holds, where $x(3,d)$ is solution to \eqref{eq:BP}. Since $(d,x)\to \Phi(d,x)$ is continuous, to prove the continuity of $\bPhi(\cdot)$, it suffices to show that $d\to x(3,d)$ is continuous. Note that the function $\psi(d,x):=\Psi_d(x)-x$ satisfies $\frac{\partial \psi}{\partial x}<0$ by Lemma \ref{lem:threederivative}. Since $x(3,d)$ is defined to be the root of $\psi(d,\cdot)$, this implies that $d\to x(3,d)$ is continuous by the implicit function theorem. Hence, we conclude that $d\to \bPhi(d)$ is continuous. Since  $\boldsymbol{\prescript{\star}{}\Phi}(d_{\textnormal{lbd}}(3))>0$ holds by Lemma \ref{lem:threephipositive} and $\boldsymbol{\prescript{\star}{}\Phi}(d_{\textnormal{ubd}}(3))<0$ holds by Lemma \ref{lem:threephinegative}, we conclude the proof.
    \end{proof}

    \section*{Acknowledgements}
     \label{sec:7}
We thank the MIT PRIMES program and its organizers Pavel Etingof, Slava Gerovitch, and Tanya Khovanova for making this possible. Y.S. thanks Elchanan Mossel, Allan Sly, and Nike Sun for encouraging feedbacks. Y.S. is supported by Simons-NSF Collaboration on Deep Learning NSF DMS-2031883 and Vannevar Bush Faculty Fellowship award ONR-N00014-20-1-2826.
    \bibliographystyle{amsalpha}

    \bibliography{bibliography}

\newcommand{\etalchar}[1]{$^{#1}$}
\providecommand{\bysame}{\leavevmode\hbox to3em{\hrulefill}\thinspace}
\providecommand{\MR}{\relax\ifhmode\unskip\space\fi MR }
\providecommand{\MRhref}[2]{%
  \href{http://www.ams.org/mathscinet-getitem?mr=#1}{#2}
}
\providecommand{\href}[2]{#2}
\begin{thebibliography}{ACOGM20}

\bibitem[AB88]{Alon88}
N.~Alon and Z.~Bregman, \emph{Every 8-uniform 8-regular hypergraph is
  2-colorable}, Graphs and Combinatorics \textbf{4} (1988), no.~1, 303--306.

\bibitem[ACIM01]{acim01}
Dimitris Achlioptas, Arthur Chtcherba, Gabriel Istrate, and Cristopher Moore,
  \emph{The phase transition in 1-in-$k$ {SAT} and {NAE} 3-sat}, Proceedings of
  the Twelfth Annual ACM-SIAM Symposium on Discrete Algorithms (Philadelphia,
  PA, USA), SODA '01, Society for Industrial and Applied Mathematics, 2001,
  pp.~721--722.

\bibitem[ACOG22]{Ayre22}
Peter Ayre, Amin Coja-Oghlan, and Catherine Greenhill, \emph{Lower bounds on
  the chromatic number of random graphs}, Combinatorica \textbf{42} (2022),
  no.~5, 617--658.

\bibitem[ACOGM20]{acgm17}
Peter Ayre, Amin Coja-Oghlan, Pu~Gao, and No{\"e}la M{\"u}ller, \emph{The
  satisfiability threshold for random linear equations}, Combinatorica
  \textbf{40} (2020), no.~2, 179--235.

\bibitem[AM02]{Achlioptas02}
Dimitris Achlioptas and Cristopher Moore, \emph{On the 2-colorability of random
  hypergraphs}, Randomization and Approximation Techniques in Computer Science
  (Berlin, Heidelberg) (Jos{\'e} D.~P. Rolim and Salil Vadhan, eds.), Springer
  Berlin Heidelberg, 2002, pp.~78--90.

\bibitem[AM06]{am06}
\bysame, \emph{Random {$k$}-{SAT}: two moments suffice to cross a sharp
  threshold}, SIAM J. Comput. \textbf{36} (2006), no.~3, 740--762. \MR{2263010}

\bibitem[AN05]{an05}
Dimitris Achlioptas and Assaf Naor, \emph{The two possible values of the
  chromatic number of a random graph}, Ann. of Math. (2) \textbf{162} (2005),
  no.~3, 1335--1351. \MR{2179732}

\bibitem[ANP05]{anp05}
Dimitris Achlioptas, Assaf Naor, and Yuval Peres, \emph{Rigorous location of
  phase transitions in hard optimization problems}, Nature \textbf{435} (2005),
  no.~7043, 759--764.

\bibitem[AP04]{ap04}
Dimitris Achlioptas and Yuval Peres, \emph{The threshold for random {$k$}-{SAT}
  is {$2^k\log 2-O(k)$}}, J. Amer. Math. Soc. \textbf{17} (2004), no.~4,
  947--973. \MR{2083472}

\bibitem[BBC{\etalchar{+}}01]{bbckw01}
B\'{e}la Bollob\'{a}s, Christian Borgs, Jennifer~T. Chayes, Jeong~Han Kim, and
  David~B. Wilson, \emph{The scaling window of the 2-{SAT} transition}, Random
  Structures Algorithms \textbf{18} (2001), no.~3, 201--256. \MR{1824274}

\bibitem[BCO16]{bc16}
Victor Bapst and Amin Coja-Oghlan, \emph{The condensation phase transition in
  the regular {$k$}-{SAT} model}, Approximation, randomization, and
  combinatorial optimization. {A}lgorithms and techniques, LIPIcs. Leibniz Int.
  Proc. Inform., vol.~60, Schloss Dagstuhl. Leibniz-Zent. Inform., Wadern,
  2016, pp.~Art. No. 22, 18. \MR{3566764}

\bibitem[BCOH{\etalchar{+}}16]{bchrv16}
Victor Bapst, Amin Coja-Oghlan, Samuel Hetterich, Felicia Ra\ss~mann, and Dan
  Vilenchik, \emph{The condensation phase transition in random graph coloring},
  Comm. Math. Phys. \textbf{341} (2016), no.~2, 543--606. \MR{3440196}

\bibitem[Bor17]{Borokov17}
A.~A. Borovkov, \emph{Generalization and refinement of the integro-local stone
  theorem for sums of random vectors}, Theory of Probability \& Its
  Applications \textbf{61} (2017), no.~4, 590--612.

\bibitem[CO13]{c13}
Amin Coja-Oghlan, \emph{Upper-bounding the {$k$}-colorability threshold by
  counting covers}, Electron. J. Combin. \textbf{20} (2013), no.~3, Paper 32,
  28. \MR{3104530}

\bibitem[COEH16]{ceh16}
Amin Coja-Oghlan, Charilaos Efthymiou, and Samuel Hetterich, \emph{On the
  chromatic number of random regular graphs}, J. Combin. Theory Ser. B
  \textbf{116} (2016), 367--439. \MR{3425250}

\bibitem[COKPZ18]{ckpz18}
Amin Coja-Oghlan, Florent Kr\c{z}aka\l{}a, Will Perkins, and Lenka
  Zdeborov\'{a}, \emph{Information-theoretic thresholds from the cavity
  method}, Adv. Math. \textbf{333} (2018), 694--795. \MR{3818090}

\bibitem[COP12]{cp12}
Amin Coja-Oghlan and Konstantinos Panagiotou, \emph{Catching the {$k$}-{NAESAT}
  threshold [extended abstract]}, S{TOC}'12---{P}roceedings of the 2012 {ACM}
  {S}ymposium on {T}heory of {C}omputing, ACM, New York, 2012, pp.~899--907.
  \MR{2961553}

\bibitem[COP13]{COGoing13}
\bysame, \emph{Going after the k-sat threshold}, Proceedings of the Forty-Fifth
  Annual ACM Symposium on Theory of Computing (New York, NY, USA), STOC '13,
  Association for Computing Machinery, 2013, p.~705–714.

\bibitem[COP16]{cp16}
\bysame, \emph{The asymptotic {$k$}-{SAT} threshold}, Adv. Math. \textbf{288}
  (2016), 985--1068. \MR{3436404}

\bibitem[COP19]{CP19}
Amin Coja-Oghlan and Will Perkins, \emph{Spin systems on {B}ethe lattices},
  Communications in Mathematical Physics \textbf{372} (2019), no.~2, 441--523.

\bibitem[COV13]{cv13}
Amin Coja-Oghlan and Dan Vilenchik, \emph{Chasing the {$k$}-colorability
  threshold}, 2013 {IEEE} 54th {A}nnual {S}ymposium on {F}oundations of
  {C}omputer {S}cience---{FOCS} '13, IEEE Computer Soc., Los Alamitos, CA,
  2013, pp.~380--389. \MR{3246240}

\bibitem[COZ12]{cz12}
Amin Coja-Oghlan and Lenka Zdeborov\'{a}, \emph{The condensation transition in
  random hypergraph 2-coloring}, Proceedings of the {T}wenty-{T}hird {A}nnual
  {ACM}-{SIAM} {S}ymposium on {D}iscrete {A}lgorithms, SODA '12, ACM, New York,
  2012, pp.~241--250. \MR{3205212}

\bibitem[CR92]{cr92}
V.~Chvatal and B.~Reed, \emph{Mick gets some (the odds are on his side)
  (satisfiability)}, Proceedings of the 33rd Annual Symposium on Foundations of
  Computer Science (Washington, DC, USA), SFCS '92, IEEE Computer Society,
  1992, pp.~620--627.

\bibitem[DFG15]{DYER15}
Martin Dyer, Alan Frieze, and Catherine Greenhill, \emph{On the chromatic
  number of a random hypergraph}, Journal of Combinatorial Theory, Series B
  \textbf{113} (2015), 68--122.

\bibitem[DGM{\etalchar{+}}10]{dgmm10}
Martin Dietzfelbinger, Andreas Goerdt, Michael Mitzenmacher, Andrea Montanari,
  Rasmus Pagh, and Michael Rink, \emph{Tight thresholds for cuckoo hashing via
  {XORSAT}}, Automata, Languages and Programming (Berlin, Heidelberg) (Samson
  Abramsky, Cyril Gavoille, Claude Kirchner, Friedhelm Meyer auf~der Heide, and
  Paul~G. Spirakis, eds.), Springer Berlin Heidelberg, 2010, pp.~213--225.

\bibitem[DM02]{dm02}
Olivier Dubois and Jacques Mandler, \emph{The 3-{XORSAT} threshold},
  Proceedings of the 43rd Symposium on Foundations of Computer Science
  (Washington, DC, USA), FOCS '02, IEEE Computer Society, 2002, pp.~769--778.

\bibitem[DRZ08]{Dall_Asta_2008}
L.~Dall'Asta, A.~Ramezanpour, and R.~Zecchina, \emph{Entropy landscape and
  non-gibbs solutions in constraint satisfaction problems}, Physical Review E
  \textbf{77} (2008), no.~3.

\bibitem[DSS14]{dss14stoc}
Jian Ding, Allan Sly, and Nike Sun, \emph{Satisfiability threshold for random
  regular nae-sat}, Proceedings of the Forty-Sixth Annual ACM Symposium on
  Theory of Computing (New York, NY, USA), STOC '14, Association for Computing
  Machinery, 2014, p.~814–822.

\bibitem[DSS15]{dss15ksat}
\bysame, \emph{Proof of the satisfiability conjecture for large k}, Proceedings
  of the Forty-seventh Annual ACM Symposium on Theory of Computing (New York,
  NY, USA), STOC '15, ACM, 2015, pp.~59--68.

\bibitem[DSS16a]{dss16maxis}
\bysame, \emph{Maximum independent sets on random regular graphs}, Acta Math.
  \textbf{217} (2016), no.~2, 263--340. \MR{3689942}

\bibitem[DSS16b]{dss16}
\bysame, \emph{Satisfiability threshold for random regular {NAE-SAT}}, Commun.
  Math. Phys. \textbf{341} (2016), no.~2, 435--489.

\bibitem[DSS22]{DSS22}
Jian Ding, Allan Sly, and Nike Sun, \emph{{Proof of the satisfiability
  conjecture for large $k$}}, Annals of Mathematics \textbf{196} (2022), no.~1,
  1 -- 388.

\bibitem[DZ10]{DZ10}
Amir Dembo and Ofer Zeitouni, \emph{Large deviations techniques and
  applications}, Stochastic Modelling and Applied Probability, vol.~38,
  Springer-Verlag, Berlin, 2010. \MR{2571413}

\bibitem[FL03]{Franz03}
Silvio Franz and Michele Leone, \emph{Replica bounds for optimization problems
  and diluted spin systems}, Journal of Statistical Physics \textbf{111}
  (2003), no.~3, 535--564.

\bibitem[GP23]{gu2023uniqueness}
Yuzhou Gu and Yury Polyanskiy, \emph{Uniqueness of bp fixed point for the potts
  model and applications to community detection}, arXiv preprint,
  arXiv:2303.14688 (2023).

\bibitem[Gue03]{guerra03}
Francesco Guerra, \emph{Broken replica symmetry bounds in the mean field spin
  glass model}, Communications in Mathematical Physics \textbf{233} (2003),
  no.~1, 1--12.

\bibitem[HY13]{Henning13}
Michael~A. Henning and Anders Yeo, \emph{2-colorings in k-regular k-uniform
  hypergraphs}, European Journal of Combinatorics \textbf{34} (2013), no.~7,
  1192--1202.

\bibitem[HY18]{HENNING18}
\bysame, \emph{Not-all-equal 3-sat and 2-colorings of 4-regular 4-uniform
  hypergraphs}, Discrete Mathematics \textbf{341} (2018), no.~8, 2285--2292.

\bibitem[JLR00]{jlrrg}
Svante Janson, Tomasz Luczak, and Andrzej Rucinski, \emph{Random graphs},
  Wiley-Interscience Series in Discrete Mathematics and Optimization,
  Wiley-Interscience, New York, 2000. \MR{1782847}

\bibitem[Kar72]{Karp72}
Richard~M. Karp, \emph{Reducibility among combinatorial problems}, pp.~85--103,
  Springer US, Boston, MA, 1972.

\bibitem[KKKS98]{kkks98}
Lefteris~M. Kirousis, Evangelos Kranakis, Danny Krizanc, and Yannis~C.
  Stamatiou, \emph{Approximating the unsatisfiability threshold of random
  formulas}, Random Structures Algorithms \textbf{12} (1998), no.~3, 253--269.
  \MR{1635256}

\bibitem[KMRT{\etalchar{+}}07]{kmrsz07}
Florent Kr\c{z}aka\l{}a, Andrea Montanari, Federico Ricci-Tersenghi, Guilhem
  Semerjian, and Lenka Zdeborov{\'a}, \emph{Gibbs states and the set of
  solutions of random constraint satisfaction problems}, Proceedings of the
  National Academy of Sciences \textbf{104} (2007), no.~25, 10318--10323.

\bibitem[MM09]{mm09}
Marc M\'{e}zard and Andrea Montanari, \emph{Information, physics, and
  computation}, Oxford Graduate Texts, Oxford University Press, Oxford, 2009.
  \MR{2518205}

\bibitem[MMZ06]{Mertens06}
Stephan Mertens, Marc M\'{e}zard, and Riccardo Zecchina, \emph{Threshold values
  of random k-sat from the cavity method}, Random Structures \& Algorithms
  \textbf{28} (2006), no.~3, 340--373.

\bibitem[MPZ02]{mpz02}
M.~M{\'e}zard, G.~Parisi, and R.~Zecchina, \emph{Analytic and algorithmic
  solution of random satisfiability problems}, Science \textbf{297} (2002),
  no.~5582, 812--815.

\bibitem[MRSY19]{montanari2019generalization}
Andrea Montanari, Feng Ruan, Youngtak Sohn, and Jun Yan, \emph{The
  generalization error of max-margin linear classifiers: Benign overfitting and
  high-dimensional asymptotics in the overparametrized regime},
  arXiv:1911.01544 (2019).

\bibitem[NSS20]{NSS}
Danny Nam, Allan Sly, and Youngtak Sohn, \emph{One-step replica symmetry
  breaking of random regular {NAE-SAT} {I}}, arXiv preprint, arXiv:2011.14270
  (2020).

\bibitem[NSS21]{nss2}
\bysame, \emph{One-step replica symmetry breaking of random regular {NAE-SAT}
  {II}}, arXiv preprint, arXiv:2112.00152 (2021).

\bibitem[NSS22]{nss22FOCS}
\bysame, \emph{One-step replica symmetry breaking of random regular nae-sat},
  2021 IEEE 62nd Annual Symposium on Foundations of Computer Science (FOCS),
  2022, pp.~310--318.

\bibitem[PS16]{ps16}
Boris Pittel and Gregory~B. Sorkin, \emph{The satisfiability threshold for
  {$k$}-{XORSAT}}, Combin. Probab. Comput. \textbf{25} (2016), no.~2, 236--268.
  \MR{3455676}

\bibitem[PT04]{PT04}
Dmitry Panchenko and Michel Talagrand, \emph{Bounds for diluted mean-fields
  spin glass models}, Probability Theory and Related Fields \textbf{130}
  (2004), no.~3, 319--336.

\bibitem[Sey74]{seymour74}
P.~D. Seymour, \emph{{On the two-coloring of hypergraphs}}, The Quarterly
  Journal of Mathematics \textbf{25} (1974), no.~1, 303--311.

\bibitem[SS23]{SS23}
Allan Sly and Youngtak Sohn, \emph{Local geometry of {NAE-SAT} solutions in the
  condensation regime}, arXiv preprint, arXiv:2305.17334 (2023).

\bibitem[SSZ16]{ssz16}
Allan Sly, Nike Sun, and Yumeng Zhang, \emph{The number of solutions for random
  regular {NAE-SAT}}, Proceedings of the 57th Symposium on Foundations of
  Computer Science, FOCS '16, 2016, pp.~724--731.

\bibitem[SSZ22]{sly_sun_zhang_2021}
\bysame, \emph{The number of solutions for random regular {NAE-SAT}},
  Probability Theory and Related Fields \textbf{182} (2022), no.~1-2, 1–109.

\bibitem[ST03]{shcherbina2003rigorous}
Mariya Shcherbina and Brunello Tirozzi, \emph{{Rigorous solution of the Gardner
  problem}}, Communications in Mathematical Physics \textbf{234} (2003), no.~3,
  383--422.

\bibitem[Tal10]{TalagrandVolI}
Michel Talagrand, \emph{Mean field models for spin glasses: Volume i},
  Springer-Verlag, Berlin, 2010.

\bibitem[YP22]{yu2022ising}
Qian Yu and Yury Polyanskiy, \emph{Ising model on locally tree-like graphs:
  Uniqueness of solutions to cavity equations}, arXiv preprint,
  arXiv:2211.15242 (2022).

\end{thebibliography}

\end{document}